\documentclass[a4paper,reqno,12pt]{amsart}
\usepackage[tmargin=40mm,bmargin=35mm,hmargin=30mm]{geometry}
\usepackage{amsmath}
\usepackage{amssymb}
\usepackage{amsthm}
\usepackage{eucal}
\usepackage{enumerate}
\usepackage{lscape}
\input xy
\xyoption{all}
\usepackage{bm}
\numberwithin{equation}{section}
\usepackage{color}

\newtheorem{theorem}{Theorem}[section]

\newtheorem{corollary}[theorem]{Corollary}
\newtheorem{lemma}[theorem]{Lemma}
\newtheorem{proposition}[theorem]{Proposition}

\theoremstyle{definition}
\newtheorem{definition}[theorem]{Definition}
\newtheorem{remark}[theorem]{Remark}
\newtheorem{example}[theorem]{Example}
\newtheorem{question}[theorem]{Question}

\newcommand{\simp}{\operatorname{\mathsf{sim}}}
\newcommand{\brick}{\mathsf{brick}\hspace{.01in}}
\newcommand{\sbrick}{\mathsf{sbrick}\hspace{.01in}}
\newcommand{\wide}{\mathsf{wide}\hspace{.01in}}
\newcommand{\FPdim}{\mathsf{FPdim}\hspace{.01in}}
\renewcommand{\top}{\operatorname{top}\nolimits}
\newcommand{\sttilt}{\operatorname{\mathsf{\tau -tiltp}}}
\newcommand{\trigidp}{\operatorname{\mathsf{\tau -rigidp}}}
\renewcommand{\mod}{\mathsf{mod}\hspace{.01in}}
\newcommand{\kD}{\mathbb{D}}
\newcommand{\add}{\mathsf{add}\hspace{.01in}}
\newcommand{\Fac}{\mathsf{Fac}\hspace{.01in}}

\renewcommand{\ker}{\mathsf{Ker}\hspace{.01in}}
\newcommand{\im}{\mathsf{Im}\hspace{.01in}}

\newcommand{\Ext}{\operatorname{Ext}\nolimits}
\newcommand{\End}{\operatorname{End}\nolimits}
\newcommand{\Hom}{\operatorname{Hom}\nolimits}
\newcommand{\rad}{\operatorname{rad}\nolimits}
\newcommand{\soc}{\operatorname{soc}\nolimits}

\newcommand{\diag}{\operatorname{diag}}

\begin{document}
\title[FP dimension via $\tau$-tilting theory]{Frobenius--Perron dimension via $\tau$-tilting theory}

\author{Takahide Adachi}
\address{T.~Adachi: Faculty of Global and Science Studies, Yamaguchi University, 1677-1 Yoshida, Yamaguchi 753-8511, Japan}
\email{tadachi@yamaguchi-u.ac.jp}

\author{Ryoichi Kase}
\address{R.~Kase: Department of Information Science and Engineering, Okayama University of Science, 1-1 Ridaicho, Kita-ku, Okayama-shi 700-0005, Japan}
\email{r-kase@ous.ac.jp}

\date{\today}

\begin{abstract}
From the perspective of $\tau$-tilting theory, we study Frobenius--Perron dimensions of finite-dimensional algebras. First, we evaluate the Frobenius--Perron dimensions of $\tau$-tilting finite algebras by a combinatorial method in $\tau$-tilting theory. Secondly, we give the upper bound for the Frobenius--Perron dimension for $\tau$-tilting finite algebras of tame representation type. Thirdly, we determine the Frobenius--Perron dimensions of Nakayama algebras and generalized preprojective algebras of Dynkin type in the sense of Gei\ss--Leclerc--Schr\"oer.
\end{abstract}
\maketitle

\section{Introduction}

The Frobenius--Perron dimension of an endofunctor of a linear category over a field was introduced in \cite{CGWZ19}. This is a generalization of the Frobenius--Perron dimension of an object in a fusion category (\cite{FK93, ENO05}). Recently, as a special case, the Frobenius--Perron dimension of an algebra $A$ is defined as
\begin{align}
\FPdim(A):=\sup\{ \rho(Q_{\mathcal{S}}) \mid \text{$\mathcal{S}$: finite semibrick in $\mod A$}\},\notag
\end{align}
where $\mod A$ denotes the category of finitely generated right $A$-modules, $Q_{\mathcal{S}}$ is the Ext-quiver of $\mathcal{S}$ and $\rho(Q_{\mathcal{S}})$ is the spectral radius of the adjacent matrix of $Q_{\mathcal{S}}$.
It is shown in \cite{CGWZ19,CGWZ23} that Frobenius--Perron dimension has a connection with representation type.
In particular, trichotomy theorem for representation type of path algebras holds.

\begin{theorem}[{\cite[Theorem 0.3]{CGWZ19}}]\label{thm:CGWZ-thm03}
Let $Q$ be a finite connected quiver and let $A$ be its path algebra over an algebraically closed field.
Then the following statements hold.
\begin{itemize}
\item[(1)] $A$ is of finite representation type if and only if $\FPdim(A)=0$.
\item[(2)] $A$ is of tame representation type if and only if $\FPdim(A)=1$.
\item[(3)] $A$ is of wild representation type if and only if $\FPdim(A)=\infty$.
\end{itemize}
\end{theorem}

The Frobenius--Perron dimension of an algebra has been calculated in some classes, such as modified ADE bound quiver algebras (\cite{W19}), representation-directed algebras (\cite{CC23}), loop-reduced algebras (\cite{CC24}).

In this paper, we study the Frobenius--Perron dimension of an algebra from the viewpoint of $\tau$-tilting theory. 
One of the remarkable results in $\tau$-tilting theory is the connection of various objects, containing semibricks, in the representation theory of algebras through $\tau$-tilting pairs.
In particular, for a $\tau$-tilting finite algebra, that is, an algebra with finitely many $\tau$-tilting pairs, semibricks bijectively correspond to $\tau$-tilting pairs (\cite{As20}). 
Due to the finiteness of the set of semibricks, the following result holds.

\begin{theorem}
If $A$ is a $\tau$-tilting finite algebra, then the Frobenius--Perron dimension of $A$ is finite.
\end{theorem}

For an algebra $A$, it is known that the set $\sttilt(A)$ of isomorphism classes of basic $\tau$-tilting pairs for $A$ admits a partial order in a natural way. 
Furthermore, if $A$ is $\tau$-tilting finite, the poset $\sttilt(A)$ forms a lattice (\cite{IRTT15}).
In \cite{Ka24}, it was shown that the Ext-quiver of a semibrick is determined by the lattice structure in $\sttilt(A)$, except for loops and multiple arrows.
Based on the idea in \cite{Ka24}, we introduce the notion of the Frobenius--Perron dimension of a finite lattice, which can be calculated in a combinatorial method.

Our first aim is to compare the Frobenius--Perron dimension $\FPdim(A)$ of a $\tau$-tilting finite algebra $A$ and the Frobenius--Perron dimension $\FPdim(\sttilt(A))$ of the finite lattice $\sttilt(A)$.

\begin{theorem}[Theorem \ref{thm:ul_bound}]
Let $A$ be a $\tau$-tilting finite algebra over an algebraically closed field $\Bbbk$.
Then we have 
\begin{align}
\max\{ \FPdim(\sttilt(A)), d_{b}\}\leq \FPdim(A) \leq\FPdim(\sttilt(A))+d_{b},\notag
\end{align}
where $d_{b}:=\max\{\dim_{\Bbbk}\Ext_{A}^{1}(X,X)\mid X: \text{brick in $\mod A$}\}$.
\end{theorem}

Our second aim is to give an upper bound for the Frobenius--Perron dimension of $\tau$-tilting finite algebras of tame representation type.
Such algebras include Brauer graph algebras with graphs having no even cycles and at most one odd cycle (see \cite{AAC18}).

\begin{theorem}[Theorem \ref{thm:ub_rf}]
Let $A$ be a $\tau$-tilting finite algebra.
If $A$ is of tame representation type, then the Frobenius--Perron dimension is at most two. 
Furthermore, if $A$ is of finite representation type, then the Frobenius--Perron dimension is less than two.
\end{theorem}

Our third aim is to determine the Frobenius--Perron dimensions of Nakayama algebras and generalized preprojective algebras, which are fundamental classes in the representation theory of algebras.

\begin{theorem}[Theorems \ref{thm:FPdim_of_Nakayama} and \ref{thm:FPdim-preproj}]\label{mainthm:FPdim_NA_PPA}
The following statements hold.
\begin{itemize}
\item[(1)] Let $A$ be a connected Nakayama algebra.
Then we have
\begin{align}
\FPdim(A)=
\begin{cases}
\ 0 &(\text{if $A$ is linear}),\\
\ 1 &(\text{if $A$ is cyclic}).
\end{cases}\notag
\end{align}
\item[(2)] Let $C$ be a Cartan matrix of a Dynkin diagram $\mathsf{X}_n$ and $D$ its symmetrizer.
Let $A:=\Pi(C,D)$ be the generalized preprojective algebra associated with $(C,D)$.
Then we have $\FPdim(A)=\rho(Q)$, where $Q$ is the Gabriel quiver of $A$.
Furthermore, the spectral radius $\rho(Q)$ is given by the following tables.
\begin{table}[ht]
\centering
\caption{$D$: minimal}
\begin{tabular}{c||c|c|c|c|c}\hline\hline
$\mathsf{X}_n$ & $\mathsf{A}_n$ & $\mathsf{B}_n$ &$\mathsf{C}_n$& $\mathsf{D}_n$ 
\\\hline
$\rho(Q)$ & $2\cos(\frac{\pi}{n+1})$ & 
 $1+2\cos(\frac{2\pi}{2n+1})$ &$2\cos(\frac{\pi}{2n+1})$ &
 $2\cos(\frac{\pi}{2(n-1)})$\\
 \hline\hline
$\mathsf{X}_n$ & $\mathsf{E}_6$ & $\mathsf{E}_7$ & $\mathsf{E}_8$&
$\mathsf{F}_4$ & $\mathsf{G}_2$\\\hline
$\rho(Q)$ & $2\cos(\frac{\pi}{12})$ & $2\cos(\frac{\pi}{18})$ & $2\cos(\frac{\pi}{30})$
 & $\frac{1+\sqrt{13}}{2}$ & $\frac{1+\sqrt{5}}{2}$
\\\hline\hline
\end{tabular}\notag
\end{table}
\begin{table}[ht]
\centering
\caption{$D$: non-minimal}
\begin{tabular}{c||c|c|c|c|c}\hline\hline
$\mathsf{X}_n$ & $\mathsf{A}_n$,$\mathsf{B}_n$,$\mathsf{C}_n$,$\mathsf{F}_{4}$,$\mathsf{G}_{2}$& $\mathsf{D}_n$ 
&$\mathsf{E}_6$ & $\mathsf{E}_7$ & $\mathsf{E}_8$
\\\hline
$\rho(Q)$ & $1+2\cos(\frac{\pi}{n+1})$& 
 $1+2\cos(\frac{\pi}{2(n-1)})$
 & $1+2\cos(\frac{\pi}{12})$ & $1+2\cos(\frac{\pi}{18})$ & $1+2\cos(\frac{\pi}{30})$
\\\hline\hline
\end{tabular}\notag
\end{table}
\end{itemize}
\end{theorem}

\subsection*{Notation and convention}
Throughout this paper, $\Bbbk$ is an algebraically closed field and $\kD:=\Hom_{\Bbbk}(-,\Bbbk)$.
By an algebra and a module, we mean a basic finite-dimensional $\Bbbk$-algebra (i.e., it is isomorphic to a bound quiver $\Bbbk$-algebra) and a finitely generated right module, respectively, unless otherwise stated.
Let $A$ be an algebra and fix a complete set $\{ e_{i}\mid i\in \Lambda\}$ of primitive orthogonal idempotents of $A$.
Let $P(i):=e_{i}A$ be an indecomposable projective $A$-module and $S(i):=\top P(i)$ a simple $A$-module.

By a quiver, we mean a finite quiver.
In this paper, we formally consider quivers whose vertex set is the empty set.
For a quiver $Q$, let $Q_{0}$ be the set of vertices and $Q_{1}$ the set of arrows. 
Let $Q^{\circ}$ denote the subquiver of $Q$ obtained by removing all loops. 
For $x\in Q_{0}$, an element $y\in Q_{0}$ is called a \emph{direct successor} (respectively, \emph{direct predecessor}) of $x$ if there exists an arrow $x\rightarrow y$ (respectively, $y\rightarrow x$) in $Q$.
For a partially ordered set (poset for short), an element is called a \emph{direct successor} (respectively, \emph{direct predecessor}) if it is a direct successor (respectively, direct predecessor) in the Hasse quiver.

\section{Preliminaries}

In this section, we quickly review the Frobenius--Perron dimension of an algebra and $\tau$-tilting theory.

\subsection{Frobenius--Perron dimension of an algebra}
In this subsection, we recall the definition and basic properties of the Frobenius--Perron dimensions of algebras. To define Frobenius--Perron dimension, we need the notion of semibricks, which plays an important role in both Frobenius--Perron dimension and $\tau$-tilting theory.

\begin{definition}
Let $A$ be an algebra.
\begin{itemize}
\item[(1)] An $A$-module $S$ is called a \emph{brick} if $\End_{A}(S)$ is a division ring.
Let $\brick(A)$ denote the set of isomorphism classes of bricks in $\mod A$. 
\item[(2)] A subset $\mathcal{S}\subseteq \brick(A)$ is called a \emph{semibrick} if $\Hom_{A}(S,S')=0$ whenever $S\not\cong S'\in \mathcal{S}$.
Let $\sbrick(A)$ denote the set of semibricks in $\mod A$. 
\end{itemize}
\end{definition}

The emptyset $\emptyset$ is viewed as a semibrick.
By Schur's lemma, a set of non-isomorphic simple $A$-modules is clearly a semibrick.
Note that a semibrick is not necessarily a finite set (for example, the path algebra of the Kronecker quiver).
A semibrick is said to be \emph{finite} if it contains only finitely many bricks. 

It is shown in \cite{R76} that semibricks can be realized as simple objects in wide subcategories of $\mod A$.
A full subcategory $\mathcal{W}$ of $\mod A$ is called a \emph{wide subcategory} of $\mod A$ if it is closed under kernels, cokernels, and extensions in $\mod A$. We write $\wide(A)$ for the set of wide subcategories of $\mod A$.
Note that wide subcategories are exact abelian subcategory of $\mod A$. 
For a wide subcategory $\mathcal{W}$ of $\mod A$, let $\simp(\mathcal{W})$ be the set of simple objects of $\mathcal{W}$. In particular, $\simp(A):=\simp(\mod A)$ coincides with the set of simple $A$-modules (up to isomorphisms).

\begin{proposition}[\cite{R76}]\label{prop:ringel}
There exists a bijection 
\begin{align}
\wide(A) \rightarrow \sbrick(A) \notag
\end{align}
given by $\mathcal{W}\mapsto \simp(\mathcal{W})$.
\end{proposition}

For a finite semibrick $\mathcal{S}$, we define a quiver $Q_{\mathcal{S}}$, called an \emph{Ext-quiver}, as follows:
the set of vertices is $\mathcal{S}$, and for $S,S'\in \mathcal{S}$, we draw $\dim_{\Bbbk}\Ext_{A}^{1}(S,S')$ arrows from $S$ to $S'$.
Clearly, the Gabriel quiver of $A$ is isomorphic to the Ext-quiver $Q_{\simp(A)}$ of $\simp(A)$.

The Frobenius--Perron dimension of an algebra $A$ is defined to be the supremum of the set of the spectral radii of the adjacent matrices of the Ext-quivers of all finite semibricks in $\mod A$.
We recall the definition of the spectral radius of the adjacent matrix of a quiver.
Let $Q$ be a quiver with non-empty vertex set and let $M(Q)$ be its adjacent matrix, that is $M(Q):=(m_{i,j})_{1\leq i,j\leq |Q_{0}|}$, where $m_{i,j}$ is the number of arrows from $i$ to $j$.
The \emph{spectral radius} $\rho(Q)$ of $Q$ is defined to be 
\begin{align}
\rho(Q):=\max\{ |r_{1}|, |r_{2}|,\ldots, |r_{n}|\} \in \mathbb{R},\notag
\end{align}
where $\{ r_{1}, r_{2},\ldots, r_{n}\}$ is the complete list of the eigenvalues of $M(Q)$.
For a quiver whose vertex set is the empty set,  the spectral radius is defined to be zero.

The following lemma is an elementary result of the spectral radius of a quiver.

\begin{lemma}[{\cite[Lemma 1.7 and Theorem 1.8]{CGWZ19}}]\label{lem:spectrad-subquiver}
Let $Q$ be a quiver.
Then the following statements hold.
\begin{itemize}
\item[(1)] If $Q'$ is a subquiver (not necessarily full) of $Q$, then we have $\rho(Q')\leq \rho(Q)$.
\item[(2)] $Q$ is acyclic if and only if $\rho(Q)=0$.
\item[(3)] If $Q$ admits connected components $Q^{1}, Q^{2}, \ldots, Q^{\ell}$, then we have 
\begin{align}
\rho(Q)=\max\{ \rho(Q^{1}), \rho(Q^{2}), \ldots, \rho(Q^{\ell})\}.\notag
\end{align}
\end{itemize}
\end{lemma}

Now, we define the Frobenius--Perron dimension of an algebra. 

\begin{definition}
The Frobenius--Perron dimension of $A$ is defined to be 
\begin{align}
\FPdim(A):=\sup\{ \rho(Q_{\mathcal{S}})\mid \mathcal{S}\in \sbrick(A):\text{finite}\}.\notag
\end{align}
\end{definition}

It is known that Frobenius--Perron dimension can be a non-integer (e.g., see Theorem \ref{mainthm:FPdim_NA_PPA}(2)).

\begin{remark}
There exists no algebra $A$ such that $0 < \FPdim(A) < 1$.
Indeed, if such an algebra $A$ exists, then each semibrick $\mathcal{S}$ satisfies  $\rho(Q_{\mathcal{S}})<1$.
Thus the absolute values of the eigenvalues of the square matrix $M:=M(Q_{\mathcal{S}})$ are less than one.
This implies that $\lim\limits_{n\to\infty} M^{n}=0$.
Since $M$ is an integer matrix, it is nilpotent.
It is known that the eigenvalues of a nilpotent matrix are always zero.
Thus we have $\rho(Q_{\mathcal{S}})=0$, and hence $\FPdim(A)=0$, a contradiction.
\end{remark}

We give a toy example of the Frobenius--Perron dimension of an algebra.

\begin{example}
Let $A=\Bbbk (1\to 2)$ be the path algebra.
Then all indecomposable $A$-modules (up to isomorphisms) are given by $P(1)$, $P(2)=S(2)$, and $S(1)$.
Furthermore, we have
\begin{align}
&\sbrick(A)=\{\{S(1),S(2)\},\{P(1)\},\{S(1)\},\{S(2)\},\emptyset\},\notag\\
&\{Q_{\mathcal{S}}\mid \mathcal{S}\in \sbrick(A)\setminus\{\emptyset\}\} =\{ \bullet\to \bullet ,\ \bullet\}.\notag
\end{align}
Since it is easily checked $\rho(\bullet\rightarrow\bullet)=\rho(\bullet)=0$, we obtain $\FPdim(A)=0$.
\end{example}

As will be seen in Example \ref{ex:fpdimsp}, the Frobenius--Perron dimension of an algebra is not necessarily equal to the spectral radius of its Gabriel quiver.

The Frobenius--Perron dimension of factor algebras of an algebra is bounded above by that of the original algebra.

\begin{lemma}[{\cite[Proposition 3.2]{CC23}}]\label{lem:fpdim-facalg}
Assume that $B$ is a factor algebra of $A$.
Then we have $\FPdim(B)\leq \FPdim(A)$.
\end{lemma}

For the convenience of the readers, we give a proof.

\begin{proof}
Since $B$ is a factor algebra of $A$, we can regard $\mod B$ as a full subcategory of  $\mod A$ and have a natural inclusion $\Ext_{B}^{1}(X,Y)\rightarrow \Ext_{A}^{1}(X,Y)$ for all $X,Y\in \mod B$. Then $\sbrick(B)$ is a subset of $\sbrick(A)$.
For each $\mathcal{S}\in \sbrick(B)$, there exists $\mathcal{S}'\in \sbrick(A)$ such that the Ext-quiver $Q_{\mathcal{S}}$ is a subquiver of $Q_{\mathcal{S}'}$.
By Lemma \ref{lem:spectrad-subquiver}, we have $\rho(Q_{\mathcal{S}})\leq \rho(Q_{\mathcal{S}'})$.
This implies that $\FPdim(B)\leq \FPdim(A)$ holds.
\end{proof}

\subsection{$\tau$-tilting theory}
In this subsection, we recall basic properties for $\tau$-tilting theory.
Let $A$ be an algebra. For an $A$-module $M$, we take a minimal projective presentation
\begin{align}
P_{1}\xrightarrow{\rho} P_{0}\rightarrow M\rightarrow 0.\notag
\end{align}
We define the \emph{Auslander--Reiten translation} $\tau M$ by the exact sequence
\begin{align}
0\rightarrow \tau M \rightarrow \nu P_{1} \xrightarrow{\nu\rho} \nu P_{0}, \notag
\end{align}
where $\nu:=\kD\Hom_{A}(-,A)$ is a Nakayama functor.
We call $M$ a \emph{$\tau$-rigid module} if it satisfies $\Hom_{A}(M,\tau M)=0$.
The following notions are basic in this paper.

\begin{definition}
Let $M$ be an $A$-module and $P$ a projective $A$-module.
The pair $(M,P)$ is called a \emph{$\tau$-rigid pair} if $M$ is $\tau$-rigid and $\Hom_{A}(P,M)=0$. A $\tau$-rigid pair $(M,P)$ is called a \emph{$\tau$-tilting pair} if $|A|=|M|+|P|$ holds, where $|X|$ denotes the number of non-isomorphic indecomposable direct summands of an $A$-module $X$.
\end{definition}

Let $\trigidp(A)$ denote the set of isomorphism classes of basic $\tau$-rigid pairs for $A$, and let $\sttilt(A)$ denote the set of isomorphism classes of basic $\tau$-tilting pairs for $A$, where a pair $(M,P)$ is said to be \emph{basic} if both $M$ and $P$ are basic and two pairs $(M,P), (M',P')$ are called \emph{isomorphic} if both $M\cong M'$ and $P\cong P'$ hold.  
For $(M,P),(M',P')\in \trigidp(A)$, we write $(M,P)\geq (M',P')$ if $\Hom_{A}(M',\tau M)=0$ and $\add P\subseteq \add P'$ hold.
By \cite[Lemma 2.25]{AIR14}, the relation $\geq$ is a partial order on $\sttilt(A)$.
Define a subset $\sttilt_{X}(A)$ of $\sttilt(A)$ as
\begin{align}
\sttilt_{X}(A):=\{ (M,P)\in \sttilt(A) \mid X\in\add(M\oplus P)\}.\notag
\end{align}

For a basic $\tau$-rigid pair $(M,P)$, there exist a maximum element $(M,P)^{+}=(M^{+},P)$, called a \emph{Bongartz completion} of $(M,P)$, and a minimum element $(M,P)^{-}=(M^{-},P^{-})$, called a \emph{co-Bongartz completion} of $(M,P)$, in $\sttilt_{(M,P)}(A)$. 
Consider an interval in $\sttilt(A)$ 
\begin{align}
\mathsf{Int}(M,P):=\{ X \in \sttilt(A) \mid (M,P)^{-}\leq X\leq (M,P)^{+}\}. \notag
\end{align}
By \cite[Theorem 4.4]{DIRRT23}, we have $\mathsf{Int}(M,P)=\sttilt_{(M,P)}(A)$.
Furthermore, it is realized as the set of isomorphism classes of $\tau$-tilting pairs for a certain algebra.

\begin{proposition}[{\cite[Theorem 3.16]{J15}}]\label{prop:jred}
Let $(M,P)$ be a $\tau$-rigid pair for $A$ and $(M^{+},P)$ its Bongartz completion.
Then we have a poset isomorphism 
\begin{align}
\sttilt_{(M,P)}(A)\cong \sttilt((\End_{A}(M^{+})/[M]))\notag
\end{align}
where $[M]$ is the idempotent of $\End_{A}(M^{+})$ corresponding to $\Hom_{A}(M^{+},M)$.
The algebra $A(M,P):=\End_{A}((M^{+})/[M])$ is called a \emph{$\tau$-tilting reduction} with respect to $(M,P)$.
\end{proposition}

An algebra $A$ is said to be \emph{$\tau$-tilting finite} if $\sttilt(A)$ is a finite set.
It is known that $\tau$-tilting finite algebras have various nice properties.

\begin{remark}\label{rem:tfin-quiver-nomult}
Let $A$ be a $\tau$-tilting finite algebra.
\begin{itemize}
\item[(1)] It is shown in \cite{As20} that there exists a bijection $\sttilt(A)\rightarrow \sbrick(A)$ given by $(M,P)\mapsto M/\rad_{\End_{A}(M)}M$. In particular, all semibricks are finite.
Note that all $\tau$-tilting pairs for $A$ can be obtained by mutations from the $\tau$-tilting pair $(A,0)$.
\item[(2)] The Gabriel quiver of $A$ has no multiple arrows.
Indeed, if it has a multiple arrows, then there exists a factor algebra such that it is isomorphic to a Kronecker algebra, which is not $\tau$-tilting finite.
By \cite[Corollary 1.9]{DIRRT23}, the class of $\tau$-tilting finite algebras is closed under taking factor algebras.
This implies that $A$ is not $\tau$-tilting finite, a contradiction.
\end{itemize}
\end{remark}

\begin{proposition}
If $A$ is $\tau$-tilting finite, then $\FPdim(A)$ is finite.
\end{proposition}

\begin{proof}
By Remark \ref{rem:tfin-quiver-nomult}(1), the set $\sbrick(A)$ is finite.
Thus we have the assertion.
\end{proof}

By \cite[Theorem 1.2]{IRTT15}, the poset $\sttilt(A)$ of a $\tau$-tilting finite algebra $A$ forms a lattice. 
Recall the definition of lattices. 
Let $(\mathbb{P},\leq)$ be a poset.
Let $x,y$ be elements in $P$.
If $\{ z\in \mathbb{P}\mid x,y\leq z\}$ admits a minimum element $x\vee y$, then it is called a \emph{join} of $x$ and $y$. Dually, we define a \emph{meet} $x\wedge y$ of $x$ and $y$.
We call $(\mathbb{P},\leq)$ a \emph{lattice} if for every $x,y\in \mathbb{P}$, the join $x\vee y$ and the meet $x\wedge y$ both exist. 
Note that, for a non-empty finite subset $\{ x_{1},x_{2}, \ldots, x_{n}\}$ of $\mathbb{P}$, the set
$\{ z\in \mathbb{P}\mid x_{1},x_{2}, \ldots, x_{n} \leq z\}$ admits a minimum element $\vee\{x_{1}, x_{2}, \ldots, x_{n}\}$. Dually, we define $\wedge\{x_{1}, x_{2}, \ldots, x_{n}\}$.
For an element $x\in \mathbb{P}$, let $\mathrm{dp}(x)$ denote the set of all direct predecessors of $x$ and let $\mathrm{ds}(x)$ denote the set of all direct successors of $x$ in $\mathbb{P}$.

The following proposition plays an important role in this paper.

\begin{proposition}[{\cite[Corollary 3.5(3)]{Ka24}}]\label{prop:kase35min}
Assume that $A$ is $\tau$-tilting finite. Fix a $\tau$-tilting pair $X_{0}$.
Then the following statements hold.
\begin{itemize}
\item[(1)] Consider a subset $\{ X_{1}, X_{2},\ldots, X_{l}\}$ of $\mathrm{dp}(X_{0})$.
Let $X$ be a maximal common direct summand of $X_{0}, X_{1}, \ldots, X_{l}$. 
Then the following statements hold.
\begin{itemize}
\item[(a)] $X$ is a $\tau$-rigid pair for $A$.
\item[(b)] $X_{0}$ is the co-Bongartz completion of $X$.
\item[(c)] $\vee\{ X_{0}, X_{1},\ldots, X_{l}\}$ is the Bongartz completion of $X$.
\end{itemize}
In particular, $X_{1}, \ldots, X_{l}$ are all direct predecessors of $X_{0}$ in $\sttilt_{X}(A)$.
\item[(2)] Consider a subset $\{ X_{1}, X_{2},\ldots, X_{l}\}$ of $\mathrm{ds}(X_{0})$.
Let $X$ be a maximal common direct summand of $X_{0}, X_{1}, \ldots, X_{l}$.
Then the following statements hold.
\begin{itemize}
\item[(a)] $X$ is a $\tau$-rigid pair for $A$.
\item[(b)] $X_{0}$ is the Bongartz completion of $X$.
\item[(c)] $\wedge\{ X_{0}, X_{1},\ldots, X_{l}\}$ is the co-Bongartz completion of $X$.
\end{itemize}
In particular, $X_{1}, \ldots, X_{l}$ are all direct successors of $X_{0}$ in $\sttilt_{X}(A)$.
\end{itemize}
\end{proposition}

Let $(M,P)$ be a $\tau$-rigid pair, $(M^{+},P)$ its Bongartz completion and $(M^{-},P^{-})$ its co-Bongartz completion.
Define a full subcategory $\mathcal{W}(M,P)$ of $\mod A$ as
\begin{align}
\mathcal{W}(M,P):= {}^{\perp}(\tau M)\cap P^{\perp}\cap M^{\perp}=\Fac(M^{+}) \cap (M^{-})^{\perp}, \notag
\end{align}
where $\Fac(M^{+})$ is a full subcategory of $\mod A$ consisting of factor modules of finite direct sums of copies of $M^{+}$.
This is a wide subcategory of $\mod A$ and is called a \emph{$\tau$-perpendicular category}.
It is known that $\tau$-perpendicular categories are realized as module categories of algebras.

\begin{proposition}[{\cite[Theorem 4.12]{DIRRT23}}]\label{prop:DIRRT-thm412}
Let $(M,P)$ be a $\tau$-rigid pair and $(M^{+},P)$ its Bongartz completion. 
Then there exists an equivalence of categories
\begin{align}
\mathcal{W}(M,P)\rightarrow \mod(A(M,P)),\notag
\end{align}
where $A(M,P)$ is a $\tau$-tilting reduction with respect to $(M,P)$.
In particular, the Ext-quiver of $\simp(\mathcal{W}(M,P))$ is isomorphic to the Gabriel quiver of $A(M,P)$.
\end{proposition}

For a $\tau$-tilting finite algebra, all wide subcategories are $\tau$-perpendicular categories.

\begin{proposition}[{\cite[Theorem 4.18]{DIRRT23}}]\label{prop:DIRRT-thm418}
Assume that $A$ is $\tau$-tilting finite.
Let $\mathcal{W}$ be a wide subcategory of $\mod A$.
Then there exists a $\tau$-rigid pair $(M,P)$ such that $\mathcal{W}=\mathcal{W}(M,P)$.
In particular, all wide subcategories of $\mod A$ are $\tau$-perpendicular categories.
\end{proposition}

Combining Propositions \ref{prop:ringel} and \ref{prop:DIRRT-thm418}, we have the following result.

\begin{corollary}\label{cor:Qs=QM}
Assume that $A$ is $\tau$-tilting finite. Then we have
\begin{align}
\{Q_{\mathcal{S}}\mid \mathcal{S}\in \sbrick(A)\}
=\{Q_{\simp(A(M,P))} \mid (M,P)\in \trigidp(A)\}.\notag
\end{align}
\end{corollary}

\begin{proof}
Let $\mathcal{S}$ be a semibrick in $\mod A$.
By Proposition \ref{prop:ringel}, there exists a wide subcategory $\mathcal{W}$ of $\mod A$ such that $\mathcal{S}=\simp(\mathcal{W})$. 
Since $A$ is $\tau$-tilting finite, it follows from Proposition \ref{prop:DIRRT-thm418} that $\mathcal{W}$ is equivalent to $\mod(A(M,P))$ for some $\tau$-rigid pair $(M,P)$.
Thus we have $Q_{\mathcal{S}}=Q_{\simp(A(M,P))}$.

Conversely, let $(M,P)$ be a $\tau$-rigid pair in $\mod A$. 
By $\simp(\mathcal{W}(M,P))\in \sbrick(A)$ and $\mathcal{W}(M,P)\cong \mod(A(M,P))$, we have $Q_{\simp (A(M,P))}=Q_{\simp(\mathcal{W}(M,P))}$.
\end{proof}

We provide an example that $\tau$-tilting theory is useful to study Frobenius--Perron dimension.

\begin{example}\label{ex:fpdimsp}
Let $A=\Bbbk Q/I$, where 
\begin{align}
Q:\xymatrix{1\ar@(dl,ul)^-{a}\ar@<1mm>[r]^-{b} & 2 \ar@<1mm>[r]^-{c} \ar@<1mm>[l]^-{e} &3 \ar@<1mm>[l]^-{d}}\notag
\end{align}
and $I=\langle a^{2}, ab, be, bcd, eb, cdcd, cde-ea \rangle$.
By \cite[Theorem 3.3]{K17}, the set of isomorphism classes of basic $\tau$-tilting pairs for $A$ is isomorphic to that for some preprojective algebra of Dynkin type $\mathsf{A}$.  Thus $A$ is $\tau$-tilting finite (see Proposition \ref{prop:miz}).
We can easily check $\rho(Q)<1.9$.
Let $M$ be the cokernel of $P(3)\xrightarrow{c}P(2)$.
Then $(M,0)$ is a $\tau$-rigid pair and the Bongartz completion $(M^{+},0)$ is given by $(P(1)\oplus P(2)\oplus M,0)$.
By Corollary \ref{cor:Qs=QM}, there exists a semibrick $\mathcal{S}$ such that $Q_{\mathcal{S}}$ is the Gabriel quiver of $A(M,0)$, which is given by 
\begin{align}
\xymatrix{1\ar@(dl,ul)\ar@<1mm>[r] & 2 \ar@(ur,dr)\ar@<1mm>[l]}\hspace{6mm}.\notag
\end{align}
Since the adjacent matrix is $\left[\begin{smallmatrix}1&1\\1&1\end{smallmatrix}\right]$, we have $\rho(Q_{\mathcal{S}})=2$.
This implies that $\FPdim(A)\geq \rho(Q_{\mathcal{S}})>\rho(Q)$.
Hence the Frobenius--Perron dimension is not necessarily equal to the spectral radius of the Gabriel quiver.
\end{example}

\section{Frobenius--Perron dimension of the $\tau$-tilting finite lattice}

In this section, we introduce the Frobenius--Perron dimension of a finite lattice, and compare the Frobenius--Perron dimension of a $\tau$-tilting finite algebra $A$ and that of the finite lattice $\sttilt(A)$.

Let $L:=(L,\leq)$ be a finite lattice and 
\begin{align}
\mathcal{U}^{+}:=\mathcal{U}^{+}(L):=\{ (x, Y) \mid x\in L, \emptyset\neq Y\subseteq \mathrm{dp}(x)\}.\notag
\end{align}
For $u=(x,Y)\in \mathcal{U}^{+}$, we define a quiver $Q(u)$ as follows:
the vertex set is equal to $Y$, and we draw a unique arrow $y\to y'$ if $y\not\in \mathrm{ds}(y\vee y')$ and $y\neq y'$.
Define the Frobenius--Perron dimension of $L$ as
\begin{align}
\FPdim(L):
&=\sup\{\rho(Q(u))\mid u\in \mathcal{U}^{+}(L)\}\notag\\ 
&=\sup\{\rho(Q(x,\mathrm{dp}(x)))\mid x\in L\setminus\{\text{maximum element}\}\}.\notag
\end{align}
As seen in the following example, the Frobenius--Perron dimension of a finite lattice can be calculated in a combinatorial method.

\begin{example}\label{ex:fpdimfinlat}
Let $L$ be a finite lattice given by the following Hasse quiver:
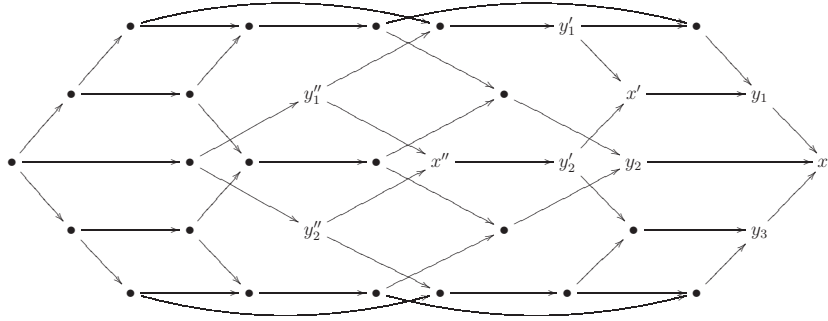
\begin{align}
\scalebox{0.6}{\xymatrix{
&&\bullet\ar[rr]\ar@/^5mm/[rrrrr]&&\bullet\ar[rr]&&\bullet\ar@/^5mm/[rrrrr]\ar[drr]&\bullet\ar[rr]&&y'_{1}\ar[rd]\ar[rr]&&\bullet\ar[dr]&&\\
&\bullet\ar[ur]\ar[rr]&&\bullet\ar[ur]\ar[dr]&&y''_{1}\ar[urr]\ar[drr]&&&\bullet\ar[drr]&&x'\ar[rr]&&y_{1}\ar[dr]&\\
\bullet\ar[ur]\ar[rrr]\ar[dr]&&&\bullet\ar[urr]\ar[drr]&\bullet\ar[rr]&&\bullet\ar[urr]\ar[drr]&x''\ar[rr]&&y'_{2}\ar[ur]\ar[dr]&y_{2}\ar[rrr]&&&x .\\
&\bullet\ar[rr]\ar[dr]&&\bullet\ar[ur]\ar[dr]&&y''_{2}\ar[urr]\ar[drr]&&&\bullet\ar[urr]&&\bullet\ar[rr]&&y_{3}\ar[ur]&\\
&&\bullet\ar[rr]\ar@/^-5mm/[rrrrr]&&\bullet\ar[rr]&&\bullet\ar[urr]\ar@/^-5mm/[rrrrr]&\bullet\ar[rr]&&\bullet\ar[ur]\ar[rr]&&\bullet\ar[ur]&&\\
&&&&&&&&&&&&&
}}\notag
\end{align}
Consider the elements $u=(x,\{y_{1},y_{2},y_{3}\})$, $u'=(x',\{y'_{1},y'_{2}\})$, and $u''=(x'',\{y''_{1},y''_{2}\})$ in $\mathcal{U}^{+}(L)$.
Then we have the following quivers respectively:
\begin{align}
Q(u)=(\xymatrix{\bullet\ar@<0.5mm>[r]\ar@/^4mm/[rr]&\bullet\ar@<0.5mm>[r]\ar@<0.5mm>[l] &\bullet\ar@<0.5mm>[l]\ar@/^4mm/[ll]}),\quad
Q(u')=(\xymatrix{\bullet\ar@/^2mm/[r] &\bullet\ar@/^2mm/[l]}),\quad
Q(u'')=(\xymatrix{\bullet& \bullet}).\notag
\end{align}
We can check that, for each $v\in \mathcal{U}^{+}$, the quiver $Q(v)$ is isomorphic to a subquiver of one of the quivers above.
By Lemma \ref{lem:spectrad-subquiver}(1), we have 
\begin{align}
\FPdim(L)=\rho(Q(u))=2,\notag
\end{align}
where the spectral radius of $Q(u)$ can be easily calculated by hand or computer. 
\end{example}

For a $\tau$-tilting finite algebra $A$, we study a relationship between $\FPdim(A)$ and $\FPdim(\sttilt(A))$.
Let us begin by comparing the Ext-quivers $Q_{\mathcal{S}}$ for $\mathcal{S}\in \sbrick(A)$ and the quivers $Q(u)$ for $u\in \mathcal{U}^{+}(\sttilt(A))$.
By \cite[Theorem 1.3]{Ka24}, the Gabriel quiver $Q$ of $A$ can be reconstructed from the lattice $\sttilt(A)$ up to loops.
For each $i\in Q_{0}$, an $A$-module $X_{i}:=(A/A(1-e_{i})A,(1-e_{i})A)$ is a $\tau$-tilting pair for $A$ and the set $\{ X_{i}\mid i\in Q_{0}\}$ coincides with the set of all direct predecessors of the $\tau$-tilting pair $(0,A)$. Thus we have 
\begin{align}
u_{0}^{A}:=((0,A),\{ X_{i}\mid i\in Q_{0}\})\in \mathcal{U}^{+}(\sttilt(A)).\notag
\end{align}
The following proposition plays an important role in this section.

\begin{proposition}[{\cite[Theorem 3.8(1)--(3)]{Ka24}}]\label{prop:thm-kase}
Assume that $A$ is $\tau$-tilting finite.
Let $Q$ be the Gabriel quiver of $A$ and $Q^{\circ}$ the subquiver of $Q$ obtained by removing all loops. Then we have $Q^{\circ}= Q(u_{0}^{A})$.
\end{proposition}

Applying the proposition above to $\tau$-tilting reduction technique by Jasso \cite{J15}, we have the following result.
Let $X$ be a $\tau$-rigid pair for $A$ and $X^{-}$ its co-Bongartz completion.
Consider the element
\begin{align}
u(X):=(X^{-},\mathrm{dp}_{X}(X^{-}))\in \mathcal{U}^{+}(\sttilt(A)),\notag
\end{align}
where $\mathrm{dp}_{X}(X^{-})$ is the set of all direct predecessors of $X^{-}$ in $\sttilt_{X}(A)$.

\begin{proposition}\label{prop:quiver-isom}
Assume that $A$ is $\tau$-tilting finite.
Let $X$ be a $\tau$-rigid pair for $A$ and let $Q$ be the Gabriel quiver of the $\tau$-tilting reduction $A(X)$ with respect to $X$.
Then we have $Q^{\circ}= Q(u(X))$.
\end{proposition}

\begin{proof}
Let $X$ be a $\tau$-rigid pair for $A$.
By Proposition \ref{prop:jred}, we have a poset isomorphism $\sttilt_{X}(A)\cong \sttilt(A(X))$, and hence $Q(u(X))\cong Q(u^{A(X)}_{0})$.
On the other hand, by Proposition \ref{prop:thm-kase}, the quiver $Q(u^{A(X)}_{0})$ is isomorphic to $Q^{\circ}$. Thus we have the assertion.
\end{proof}

Combining Corollary \ref{cor:Qs=QM} and Proposition \ref{prop:quiver-isom}, we have the desired result.

\begin{proposition}\label{prop:kj-quiver}
Assume that $A$ is $\tau$-tilting finite.
Then the following statements hold.
\begin{itemize}
\item[(1)] $\mathcal{U}^{+}(\sttilt(A))=\{ u(X)\mid X\in \trigidp(A)\}$.
\item[(2)] $\{Q_{\mathcal{S}}^{\circ}\mid \mathcal{S}\in \sbrick(A)\}=\{Q(u)\mid u\in \mathcal{U}^{+}(\sttilt(A))\}$.
\end{itemize}
\end{proposition}
\begin{proof}
(1) By definition, we have $u(X)\in \mathcal{U}^{+}(\sttilt(A))$ for each $\tau$-rigid pair $X$.
Consider an element $u:=(X_{0}, \{ X_{1}, \ldots, X_{l}\})\in \mathcal{U}^{+}(\sttilt(A))$.
By Proposition \ref{prop:kase35min}(1), there exists a $\tau$-rigid pair $X$ such that $u=u(X)$.
Therefore we have the assertion.

(2) The assertion follows from 
\begin{align}
\{ Q^{\circ}_{\mathcal{S}}\mid \mathcal{S}\in \sbrick(A)\}
&=\{Q_{\simp(A(X))}^{\circ}\mid X\in \trigidp(A)\}&&\text{by Corollary \ref{cor:Qs=QM}}\notag\\
&=\{ Q(u(X))\mid X\in \trigidp(A)\}&&\text{by Proposition \ref{prop:quiver-isom}}\notag\\
&=\{ Q(u)\mid u\in \mathcal{U}^{+}(\sttilt(A))\}&&\text{by (1)}\notag
\end{align}
The proof is complete.
\end{proof}

The following theorem is one of main results in this paper.

\begin{theorem}\label{thm:ul_bound}
Assume that $A$ is a $\tau$-tilting finite $\Bbbk$-algebra.
Then we have
\begin{align}
\max\{\FPdim(\sttilt(A)), d_{b}\} \le \FPdim(A) \le \FPdim(\sttilt(A))+d_{b}, \notag
\end{align}
where $d_{b}:=\max\{\dim_{\Bbbk}\Ext_{A}^{1}(S,S)\mid S\in \brick(A)\}$.
In particular, if each brick has no non-trivial self-extension, then we have
\begin{align}
\FPdim(A) = \FPdim(\sttilt(A)). \notag
\end{align}
\end{theorem}
\begin{proof}
By the definition of $d_{b}$, there exists a brick $S$ such that $\dim_{\Bbbk}\Ext_{A}^{1}(S,S)=d_{b}$. Thus we obtain $\rho(Q_{\{S\}})=d_{b}$, and hence $d_{b}\leq\FPdim(A)$.
Let $\mathcal{S}$ be an arbitrary semibrick in $\mod A$ and $Q_{\mathcal{S}}$ its Ext-quiver.
By Proposition \ref{prop:kj-quiver}(2), there exists $u\in \mathcal{U}^{+}(\sttilt(A))$ such that $Q_{\mathcal{S}}^{\circ}=Q(u)$.
Let $Q(u)^{d_{b}}$ be the quiver constructed from $Q(u)$ adding $d_{b}$ loops for each vertex. 
By maximality of $d_{b}$, the quiver $Q_{\mathcal{S}}$ is a subquiver of $Q(u)^{d_{b}}$.
Thus it follows from Lemma \ref{lem:spectrad-subquiver}(1) that 
\begin{align}
\rho(Q(u))\leq \rho(Q_{\mathcal{S}})\leq \rho(Q(u)^{d_{b}})=\rho(Q(u))+d_{b}.\notag
\end{align} 
This implies that $\FPdim(\sttilt(A))\leq \FPdim(A)\leq \FPdim(\sttilt(A))+d_{b}$.
\end{proof}

The finite lattice $L$ in Example \ref{ex:fpdimfinlat} is realized as the lattice $\sttilt(A)$ for some $\tau$-tilting finite algebra $A$. Thus we have the following result.

\begin{example}
Let $A$ be a radical square zero algebra whose Gabriel quiver is
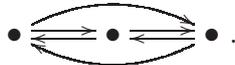
\begin{align}
\xymatrix{\bullet \ar@<0.5mm>[r] \ar@/^4mm/[rr] & \bullet \ar@<0.5mm>[r] \ar@<0.5mm>[l]&\bullet \ar@<0.5mm>[l] \ar@/^4mm/[ll]}.\notag
\end{align}
Note that $A$ is $\tau$-tilting finite (for a criterion of $\tau$-tilting finiteness, see \cite{Ad16} or \cite{Ao22}).
Then $\sttilt(A)$ is isomorphic to the finite lattice $L$ in Example \ref{ex:fpdimfinlat}. 
Thus $\FPdim(\sttilt A)=2$.
By Theorem \ref{thm:ul_bound}, we have 
\begin{align}
\max\{2, d_{b}\}\leq \FPdim(A) \leq 2+d_{b}. \notag
\end{align}
We can check that each brick has no non-trivial self extension, that is, $d_{b}=0$.
Hence we have $\FPdim(A)=2$.
\end{example}

In the following, we give three applications of Theorem \ref{thm:ul_bound}: multiplicity-free Brauer tree algebras, cluster tilted algebras, and a generalization of \cite[Theorem 4.1]{CC24}.

A multiplicity-free Brauer tree algebra is defined by a finite graph that is a tree.
For the definition and basic results, see \cite{Al86}.

\begin{proposition}
\label{proposition:BTA}
Let $A$ be a multiplicity-free Brauer tree algebra.
Then we have 
\begin{align}
\FPdim(A)=\FPdim(\sttilt(A)).\notag
\end{align}
\end{proposition}

\begin{proof}
Note that (multiplicity-free) Brauer tree algebras are of finite representation type, and hence $\tau$-tilting finite.
By Theorem \ref{thm:ul_bound}, it is enough to show that $d_{b}=0$.
Let $S$ be an arbitrary brick in $\mod A$.
If $S$ is $\tau$-rigid, then we obtain $\Ext_{A}^{1}(S,S)=0$ by the Auslander--Reiten formula.
Assume that $S$ is not $\tau$-rigid. 
By \cite[Theorem 1]{AZ14}, there exists an indecomposable projective module $P$ such that $S\cong P/\soc P$. 
Applying $\Hom_{A}(-,S)$ to the short exact sequence $0\rightarrow \soc P \rightarrow P \rightarrow S \rightarrow 0$ yields an exact sequence 
\begin{align}
\Hom_{A}(\soc P, S)\rightarrow \Ext_{A}^{1}(S,S)\rightarrow \Ext_{A}^{1}(P,S)=0.\notag
\end{align}
By $\Hom_{A}(\soc P,S)=0$, we have $\Ext_{A}^{1}(S,S)=0$.
This implies that $d_{b}=0$.
\end{proof}

We propose a naive question.
We define two special classes of Brauer tree algebras.
A graph $G$ is called a \emph{star} (respectively, a \emph{line}) if it is isomorphic to the following left-hand (respectively, right-hand) side graph:
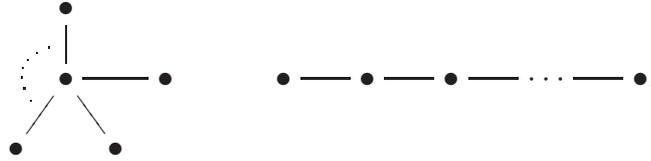
\begin{align}
\xymatrix@C=2mm@R=5mm{
&     &\bullet&     &     && &     &&     &&     &&      &&\\
&     &\bullet\ar@{-}[ld]_{}="a"\ar@{-}[u]^{}="b"\ar@{-}[rd]\ar@{-}[rr]&     &\bullet&& &\bullet\ar@{-}[rr]&&\bullet\ar@{-}[rr]&&\bullet\ar@{-}[rr]&&\cdots\ar@{-}[rr]&&\bullet .\\
&\bullet&     &\bullet&     && &     &&     &&     &&      &&
\ar@{.}@/^1pc/"a";"b"
}\notag
\end{align}
We call an algebra $A$ a \emph{Brauer star algebra} (respectively, a \emph{Brauer line algebra}) if it is a Brauer tree algebra defined by a star (respectively, a line).
It is known that Brauer star algebras are symmetric Nakayama algebras, and the Gabriel quivers of Brauer line algebras are the following forms:
\begin{align}\label{eq:An-dquiver}
\xymatrix{1\ar@<1mm>[r]&2\ar@<1mm>[l]\ar@<1mm>[r]&\cdots\ar@<1mm>[l]\ar@<1mm>[r]&n-1\ar@<1mm>[l]\ar@<1mm>[r]&n\ar@<1mm>[l]}.
\end{align}
As will be shown in Theorem \ref{thm:FPdim_of_Nakayama}(2) and Proposition \ref{prop:spect-rad-dynkin}, the Frobenius--Perron dimension of a Brauer star algebra is exactly one, and that of a Brauer line algebra is at least $2\cos(\frac{\pi}{n+1})$, where $n$ is the number of non-isomorphic simple modules.

\begin{question}
Let $G$ be a tree with $n$ edges.
Let $A_{G}$ be the (multiplicity-free) Brauer tree algebra associated with $G$.
Does the following inequalities hold?
\begin{align}
\FPdim(A_{\mathrm{star}})\leq \FPdim(A_{G})\leq \FPdim(A_{\mathrm{line}}),\notag
\end{align}
where $A_{\mathrm{star}}$ is the (multiplicity-free) Brauer star algebra and $A_{\mathrm{line}}$ is the (multiplicity-free) Brauer line algebra with $n$ non-isomorphic simple modules.
\end{question}

Next, we show that a similar statement to Proposition \ref{proposition:BTA} holds for cluster tilted algebras of Dynkin type. Let $\mathcal{C}_{Q}$ be the cluster category of finitely generated $\Bbbk Q$-modules with shift functor $[1]$. An object $T\in \mathcal{C}_{Q}$ is said to be \emph{cluster tilting} if $\add T=\{ X\in \mathcal{C}_{Q}\mid \Hom_{\mathcal{C}_{Q}}(T,X[1])=0\}$.
Note that $\mathcal{C}_{Q}$ always admits a cluster tilting object. 
An algebra $A$ is called a \emph{cluster tilted algebra} if there exists a cluster tilting object $T$ in $\mathcal{C}_{Q}$ such that $A\cong \End_{\mathcal{C}_{Q}}(T)$. 
For detail, see \cite{BMRRT06}.

\begin{proposition}\label{proposition:CTA}
Let $Q$ be a Dynkin quiver, $\mathcal{C}_Q$ the cluster category,  
$T$ a cluster tilting object in $\mathcal{C}_Q$, and $A:=\End_{\mathcal{C}_Q}(T)$ the cluster tilted algebra.
Then we have
\begin{align}
\FPdim(A)=\FPdim(\sttilt(A)).\notag
\end{align}
\end{proposition}

\begin{proof}
Fix a cluster tilting object $T$ in $\mathcal{C}_{Q}$.
By \cite[Proposition 2.1]{BMR07}, the functor $F:=\Hom_{\mathcal{C}_{Q}}(T,-):\mathcal{C}_{Q}\rightarrow \mod A$ is full.
Let $U$ be a cluster tilting object in $\mathcal{C}_{Q}$.
Decompose $U$ as $U=U'\oplus U''$, where $U''$ is a maximal direct summand of $U$ which belongs to $\add(T[1])$.
By \cite[Theorem 4.1]{AIR14}, the assignment $U\mapsto (F(U),F(U''[-1]))$ gives a bijection from the set of isomorphism classes of basic cluster tilting objects in $\mathcal{C}_{Q}$ to the set of isomorphism classes of basic $\tau$-tilting pairs for $A$. 
Since $Q$ is a Dynkin quiver, the cluster category $\mathcal{C}_{Q}$ is of finite type, 
and hence the set of isomorphism classes of basic cluster tilting objects in $\mathcal{C}_{Q}$ is finite.
This implies that $A$ is $\tau$-tilting finite.

Let $(M,P)$ be a $\tau$-rigid pair, $(M^{+},P)$ its Bongartz completion, and $A(M,P)$ its $\tau$-tilting reduction.
Then there exists a cluster tilting object $U$ in $\mathcal{C}_{Q}$ such that $(M^{+},P)=(F(U),F(U''[-1]))$.
Since the functor $F$ is full, $A(M,P)$ is a factor algebra of $\End_{C_{Q}}(U)$. 
By \cite[Corollary 6.15]{BMRRT06}, $\End_{C_{Q}}(U)$ has no loops, and so does $A(M,P)$.
Thus it follows from Proposition \ref{prop:DIRRT-thm412} that 
\begin{align}
d_{b}:=\max\{\dim_{\Bbbk}\Ext_{A}^{1}(S,S)\mid S\in \brick(A)\}=0.\notag
\end{align}
By Theorem \ref{thm:ul_bound}, we have the assertion.
\end{proof}

Using the following result in \cite{EJR18}, we give a generalization of \cite[Theorem 4.1]{CC24}.
Let $Z(A)$ be the center of $A$ and $J(A)$ the Jacobson radical of $A$.

\begin{proposition}[{\cite[Theorem 11]{EJR18}}]\label{EJR:thm11}
Let $A$ be an arbitrary algebra and let $r\in Z(A) \cap J(A)$.
Then there exists a poset isomorphism 
\begin{align}
\sttilt(A) \rightarrow \sttilt(A/\langle r\rangle),\notag
\end{align}
where $\langle r \rangle$ is the two-sided ideal of $A$ generated by $r$.
In particular, if $A$ is $\tau$-tilting finite, then $\FPdim(\sttilt(A))=\FPdim(\sttilt(A/\langle r\rangle))$.
\end{proposition}

Under a certain condition for $r\in Z(A)\cap J(A)$, we study a relationship between $\FPdim(A)$ and $\FPdim(A/\langle r\rangle)$. In the following, we fix a complete set $\{ e_{i}\mid i\in \Lambda\}$ of primitive orthogonal idempotents of $A$.

\begin{proposition}\label{prop:brick-selfext}
Let $r\in Z(A)\cap J(A)$ and $B:=A/\langle r \rangle$.
Then the following statements hold.
\begin{itemize}
\item[(1)] $\brick(A) = \brick(B)$
\item[(2)] Assume that $re_{i}=0$ holds for some $i\in \Lambda$.
If $S$ is a brick in $A$ with $Se_{i}\neq 0$, then $\Ext_{A}^{1}(S,S)=\Ext_{B}^{1}(S,S)$ holds.
\end{itemize}
\end{proposition}
\begin{proof}
Since $B$ is a factor algebra of $A$, we regard $\mod B$ as the full subcategory of $\mod A$.
Thus we have $\Hom_{B}(M,N)=\Hom_{A}(M,N)$ and $\Ext_{B}^{1}(M,N)\subseteq \Ext_{A}^{1}(M,N)$ for $M,N\in \mod B$.

(1) This follows from \cite[Proposition 2.29]{As20}. 

(2) Let $S$ be a brick in $\mod A$. By (1), it is a $B$-module, that is $Sr=0$.
We show that $\Ext_{A}^{1}(S,S)\subseteq \Ext_{B}^{1}(S,S)$.
Consider a short exact sequence 
\begin{align}
0\rightarrow S \xrightarrow{f} M\xrightarrow{g} S\rightarrow 0 \notag
\end{align}
in $\mod A$.
It is enough to show that $M$ is a $B$-module or equivalently, $Mr=0$.
By $g(Mr)=g(M)r\subset Sr=0$, we have $Mr\subset \ker g= \im f$.
Thus we can identify $Mr$ with a submodule of $S$.
By $f(S)r=f(Sr)=0$, the universality of cokernels yields an epimorphism $h: S\rightarrow Mr$.
Then $h$ is not an isomorphism. 
Indeed, if $h$ is an isomorphism, then it induces an isomorphism $Se_{i} \simeq (Mr)e_{i}$ .
By $Se_{i}\neq 0$ and $Mre_{i}=0$, this is a contradiction.
Since $S$ is a brick, the morphism $S\xrightarrow{h}Mr\subset S$ is zero.
Thus we obtain $h=0$, and hence $Mr=0$.
\end{proof}

As an immediate result, self-extensions of almost all bricks in $\mod A$ can be controlled by bricks in $\mod(A/\langle r\rangle)$ for $r\in Z(A)\cap J(e_{j}Ae_{j})$.

\begin{corollary}\label{cor:brick-ext}
Let $r\in Z(A)\cap J(e_{j}Ae_{j})$ for some $j\in \Lambda$ and $B:=A/\langle r\rangle$.
Then $\Ext_{A}^{1}(S,S)=\Ext_{B}^{1}(S,S)$ holds for all $S\in \brick(A)\setminus\{ S(j)\}$.
\end{corollary}
\begin{proof}
Let $S\in \brick A \setminus\{ S(j)\}$.
Then there exists $i\in \Lambda$ such that $i\neq j$ and $Se_{i}\neq 0$.
By $i\neq j$, we have $re_{i}=0$.
Thus the assertion follows from Proposition \ref{prop:brick-selfext}.
\end{proof}

The following results give a generalization of \cite[Theorem 4.1]{CC24}.

\begin{proposition}\label{prop:fp-reduction}
Assume that $A$ is $\tau$-tilting finite.
Let $r\in Z(A)\cap J(e_{j}Ae_{j})$ for some $j\in \Lambda$ and $B:=A/\langle r\rangle$.
If $\Ext_{B}^{1}(S,S)=0$ holds for all $S\in \brick(B)$, then we have
\begin{align}
\max\{ \FPdim(B),d \}\leq \FPdim(A)\leq \FPdim(B)+d, \notag
\end{align}
where $d:=\dim_{\Bbbk}\Ext_{A}^{1}(S(j),S(j))$.
In particular, if $\FPdim(B)=0$, then $\FPdim(A)=d$.
\end{proposition}
\begin{proof}
By Theorem \ref{thm:ul_bound}, we have
\begin{align}
\max\{\FPdim(\sttilt(A)), d_{b}\}\leq \FPdim(A)\leq \FPdim(\sttilt(A))+d_{b}. \notag
\end{align}
Since $r$ is in $Z(A)\cap J(A)$, it follows from Proposition \ref{EJR:thm11} that $\FPdim(\sttilt(A))=\FPdim(\sttilt(B))$.
Since each brick in $\mod B$ has no non-trivial self-extension, we have $\FPdim(\sttilt(B))=\FPdim(B)$ by Theorem \ref{thm:ul_bound}.
On the other hand, by Corollary \ref{cor:brick-ext}, we have $\Ext_{A}^{1}(S,S)=\Ext_{B}^{1}(S,S)=0$ for all $S\in \brick (A)\setminus\{ S(j)\}$. 
This implies that $d_{b}=d$.
The proof is complete.
\end{proof}

We immediately obtain the following result.

\begin{corollary}\label{cor:seq-fpdim}
Let $(A,B)$ be a pair of $\tau$-tilting finite algebras so that
there exists a finite sequence of ($\tau$-tilting finite) algebras
\begin{center}
$A=:A_{0}$, $A_{1}:=A_{0}/\langle r_{0}\rangle$, $\ldots$, $A_{\ell}:=A_{\ell-1}/\langle r_{\ell-1}\rangle\cong B$,
\end{center}
where $r_{i}\in Z(A_{i})\cap J(e_{j_{i}}A_{i}e_{j_{i}})$ for some $j_{i}\in \Lambda$. 
If $\Ext_{B}^{1}(S,S)=0$ holds for all $S\in \brick(B)$, then we have 
\begin{align}
\max\{ \FPdim(B), d\}\leq \FPdim(A)\leq \FPdim(B)+d, \notag
\end{align}
where $d:=\max\{ \dim_{\Bbbk}\Ext_{A}^{1}(S(j_{i}),S(j_{i}))\mid i\in\{ 0,1,\dots,\ell-1 \}\}$.
In particular, if $\FPdim(B)=0$, then $\FPdim(A)=d$.
\end{corollary}
\begin{proof}
By an argument similar to the proof in Proposition \ref{prop:fp-reduction}, we have
\begin{align}
\FPdim(\sttilt(A_{0}))=\FPdim(\sttilt(A_{1}))=\cdots =\FPdim(\sttilt(A_{\ell}))\notag
\end{align}
and 
\begin{align}
\Ext_{A_{0}}^{1}(S,S)=\Ext_{A_{1}}^{1}(S,S)=\cdots=\Ext_{A_{\ell}}^{1}(S,S)=0\notag
\end{align}
for all $S\in \brick(A) \setminus\{ S(j_{i})\mid i\in \{0,1,\dots,\ell-1\}\}$.
Thus the assertion follows from Theorem \ref{thm:ul_bound}.
\end{proof}

We give an example of $\tau$-tilting finite algebras with Frobenius--Perron dimension zero.
Let $A$ be an algebra.
A \emph{path} in $\mod A$ is a sequence 
\begin{align}
M_{0}\xrightarrow{f_{1}}M_{1}\xrightarrow{f_{2}}M_{2}\rightarrow \cdots \rightarrow M_{t-1}\xrightarrow{f_{t}}M_{t}\notag
\end{align}
of non-zero non-isomorphisms $f_{1},\ldots, f_{t}$ between indecomposable modules $M_{0},\ldots, M_{t}$ with $t\geq 1$.
Furthermore, if $M_{0}$ is isomorphic to $M_{t}$, then the path in $\mod A$ is called a \emph{cycle}.  
We call $A$ a \emph{representation-directed algebra} if there exists no cycle in $\mod A$. 
Note that all representation-directed algebras are of finite representation type (for example, see \cite[IX.3.4 Corollary]{ASS06}), and hence $\tau$-tilting finite.
It is known that representation-directed algebras have the Frobenius--Perron dimension zero (\cite[Corollary 4.3]{CC23}). 

\begin{proposition}\label{prop:cc23-cor4.3}
Let $A$ be a representation-directed algebra.
Then the following statements hold.
\begin{itemize}
\item[(1)] The Gabriel quiver $Q$ of $A$ is acyclic. In particular, $\rho(Q)=0$.
\item[(2)] If there exists a fully faithful functor $\mod B\rightarrow \mod A$, then the algebra $B$ is also representation-directed.
\item[(3)] The Frobenius--Perron dimension of a factor algebra of $A$ is zero.   
\end{itemize}
\end{proposition}

For the convenience of the readers, we give a proof by $\tau$-tilting theory.

\begin{proof}
(1) Suppose to the contrary that the Gabriel quiver $Q$ of $A$ is not acyclic.
Then there exists a path $v_{1}\rightarrow v_{2}\rightarrow\cdots\rightarrow v_{j}$ in $Q_{A}$ such that $v_{1}=v_{j}$. This yields a cycle in $\mod A$ that is given by the corresponding indecomposable projective $A$-modules, a contradiction.
By Lemma \ref{lem:spectrad-subquiver}(2), we have $\rho(Q)=0$.

(2) If $B$ is not representation-directed, then $\mod B$ admits a cycle.
By the assumption, the cycle in $\mod B$ induces a cycle in $\mod A$.
Thus $A$ is not representation-directed.

(3) By Lemma \ref{lem:fpdim-facalg}, it is enough to check that $\FPdim(A)=0$.
Let $\mathcal{S}$ be a (finite) semibrick in $\mod A$ and $\mathcal{W}$ the corresponding wide subcategory by Proposition \ref{prop:ringel}.
Since $A$ is $\tau$-tilting finite, it follows from Proposition \ref{prop:DIRRT-thm418} that $\mathcal{W}$ can be realized as a module category of some algebra $B$, that is, $\mathcal{W}\cong \mod B$.
By (2), $B$ is also representation-directed.
By (1), we have $\rho(Q_{\mathcal{S}})=\rho(Q_{\simp(B)})=0$, and hence $\FPdim(A)=0$.
\end{proof}

We explain the relationship between Corollary \ref{cor:seq-fpdim} and \cite[Theorem 4.1]{CC24} below.

\begin{remark}
The pair $(A,B)$ appeared in \cite[Theorem 4.1]{CC24} satisfies the assumption in Corollary \ref{cor:seq-fpdim}. Indeed, if $A$ is a bound quiver algebra satisfying the commutativity condition of loops and $B$ is the loop-reduced algebra of $A$ (for definitions, see \cite[Definition 3.1]{CC24}), then we can check that there exists a finite sequence of algebras from $A$ to $B$ satisfying the condition in Corollary \ref{cor:seq-fpdim}.
Since representation-directed algebras are of finite representation type, $B$ is $\tau$-tilting finite, and hence so is $A$ by Proposition \ref{EJR:thm11}.
Furthermore, by Proposition \ref{prop:cc23-cor4.3}, the Frobenius--Perron dimensions of representation-directed algebras are always zero.
Thus \cite[Theorem 4.1]{CC24} is derived from Corollary \ref{cor:seq-fpdim}.
\end{remark}

The following example cannot be covered by \cite[Theorem 4.1]{CC24}.

\begin{example}
Consider the quiver
\begin{align}
\xymatrix{
&& 2 \ar[dr]_-{y} & \\
Q:&1\ar@(ul,ur)^-{a}\ar@(ld,lu)^-{b}\ar@(dr,dl)^-{c}\ar[ur]_-{x} \ar[dr]^-{z}& & 4. \\
&& 3 \ar[ur]^-{w} &
}\notag
\end{align}
Let $L$ be the two-sided ideal of $\Bbbk Q$ generated by all loops on $Q$ and let $I$ be the two-sided ideal of $\Bbbk Q$ generated by 
\begin{align}
a^{2}+bc-cb, ab, ac, ax, az, ba, b^{2}, bx,bz, ca, c^{2},cx,cz, xy-zw. \notag
\end{align}
Then $(\Bbbk Q/I,\Bbbk Q/L)$ satisfies the assumption in Corollary \ref{cor:seq-fpdim}.
Indeed, we have a sequence of algebras
\begin{align}
A_{0}=\Bbbk Q/I, A_{1}=A_{0}/\langle a \rangle, A_{2}=A_{1}/\langle b \rangle, A_{3}=A_{2}/\langle c \rangle\cong \Bbbk Q/L. \notag
\end{align}
Furthermore, we can check that $\Bbbk Q/L$ is representation-directed.
Thus we have
\begin{align}
\FPdim(\Bbbk Q/I)=3. \notag
\end{align}
Note that the two-sided ideal $I$ does not satisfy the commutativity condition of loops.
\end{example}

\section{Frobenius--Perron dimensions of $\tau$-tilting finite algebras of tame representation type}

In this section, we determine the upper bounds of the Frobenius--Perron dimensions of algebras of finite representation type and $\tau$-tilting finite algebras of tame representation type.

\begin{theorem}\label{thm:ub_rf}
The following statements hold.
\begin{itemize}
\item[(1)] If $A$ is an algebra of finite representation type, then we have $\FPdim(A)<2$.
\item[(2)] If $A$ is a $\tau$-tilting finite algebra of tame representation type, then we have $\FPdim(A)\leq 2$.
\end{itemize}
\end{theorem}
 
As shown in the following example, the upper bounds in Theorem \ref{thm:ub_rf} are best possible.

\begin{example}
It is known that Brauer graph algebras are of finite representation type or of tame representation type. 
Furthermore, a Brauer graph algebra is of finite representation type if and only if it is a Brauer tree algebra. For example, see \cite[Corollary 2.9]{Sc17}.
\begin{itemize}
\item[(1)] Let $BL_{n}$ be a Brauer line algebra with $n$ non-isomorphic simple modules.
Then $BL_{n}$ is of finite representation type and the Gabriel quiver $Q$ of $BL_{n}$ is given by \eqref{eq:An-dquiver}.
As will be shown in Proposition \ref{prop:spect-rad-dynkin}, we have 
\begin{align}
\FPdim(BL_{n})\geq \rho(Q)= 2\cos\left(\frac{\pi}{n+1}\right).\notag
\end{align}
This implies that $\displaystyle\lim_{n\rightarrow \infty}\FPdim(BL_{n})=2$.
\item[(2)] Consider the bound quiver algebra $A:=\Bbbk Q/I$, where
\begin{align}
Q=\xymatrix{1\ar@(ul,dl)_{a}\ar@(ur,dr)^{b}}\quad\text{and}\quad I=\langle a^{2}, b^{2}, ab-ba \rangle.\notag
\end{align}
Since $A$ is a Brauer graph algebra but not a Brauer tree algebra, it is of tame representation type. 
Furthermore, we have $\sttilt(A)=\{ A, 0 \}$, and hence $A$ is $\tau$-tilting finite. 
Then we can easily check that $\FPdim(A)=\rho(Q)=2$. 
\end{itemize}
\end{example}

In the following, we prove Theorem \ref{thm:ub_rf}.
A quiver $\Delta$ is said to be \emph{bipartite} if each vertex in $\Delta$ is a sink or a source. 
Let $\Delta_{0}^{+}$ be the set of sources in $\Delta$ and $\Delta_{0}^{-}$ the set of sinks in $\Delta$. 
For a vertex $i\in \Delta_{0}$, let $\mathrm{ds}(i)$ be the set of all direct successors of $i$ in $\Delta$ and $\mathrm{dp}(i)$ the set of all direct predecessors of $i$ in $\Delta$.
If $\Delta$ is bipartite, for each $i\in \Delta_{0}^{+}$ (respectively, $j\in \Delta_{0}^{-}$), we have $\mathrm{ds}(i)\subseteq \Delta_{0}^{-}$ and $\mathrm{dp}(i)=\emptyset$ (respectively, $\mathrm{ds}(j)=\emptyset$ and $\mathrm{dp}(j)\subseteq \Delta_{0}^{+}$).
For a bipartite quiver $\Delta$, we define a quadratic form $q_{\Delta}: \mathbb{R}^{\Delta_{0}^{-}}\to \mathbb{R}$ as 
\begin{align}
q_{\Delta}(\bm{x}):=4\sum_{j\in \Delta_{0}^-}x_j^2-\sum_{i\in \Delta_0^+}(\sum_{j\in \Delta_{0}^-}a_{ij}x_{j})^{2},\notag
\end{align}
where $a_{ij}$ is the number of arrows from $i$ to $j$.
Then $q_{\Delta}$ satisfies the following property, which plays an important role in the proof of Theorem \ref{thm:ub_rf}.

\begin{lemma}\label{lem:quad-form}
Let $\Delta$ be a bipartite quiver and $q_{\Delta}$ the quadratic form.
Then the following statements hold.
\begin{itemize}
\item[(1)] If the underlying graph of $\Delta$ is a simply-laced Dynkin diagram, then $q_{\Delta}$ is positive definite, that is, $q_{\Delta}(\bm{x})>0$ holds for all $\bm{x}\neq \bm{0}$.
\item[(2)] If the underlying graph of $\Delta$ is a simply-laced extended Dynkin diagram, then $q_{\Delta}$ is positive semidefinite, that is, $q_{\Delta}(\bm{x})\geq 0$ for all $\bm{x}$, but not positive definite.
\end{itemize}
\end{lemma}

We will give the proof of Lemma \ref{lem:quad-form} at the end of this section.

In the rest of this section, we assume that $A$ is $\tau$-tilting finite.
By Remark \ref{rem:tfin-quiver-nomult}, the Gabriel quiver of $A$ has no multiple arrows.
As a result of Lemma \ref{lem:quad-form}, we have the upper bound for the spectral radius of a bipartite quiver, called a \emph{separated quiver}. 
For a quiver $Q=(Q_{0},Q_{1})$, let $Q_{0}^{\mathrm{so}}:=\{ i^{+}\mid i\in Q_{0} \}$ and $Q_{0}^{\mathrm{si}}:= \{ i^{-}\mid i\in Q_{0}\}$.
Define a separated quiver $Q^{s}=(Q_{0}^{s},Q_{1}^{s})$ of $Q$ as $Q_{0}^{s}:=Q_{0}^{\mathrm{so}}\sqcup Q_{0}^{\mathrm{si}}$ and $Q_{1}^{s}:=\{ \alpha^s\mid \alpha\in Q_1\}$ where $\alpha^s$ is an arrow from $i^{+}$ to $j^{-}$ for each arrow $\alpha:i\to j$ in $Q_1$.
Note that separated quivers are bipartite but not necessarily connected.

\begin{proposition}\label{prop:sep-quiv}
Let $Q$ be a quiver without multiple arrows and let $Q^{s}$ be its separated quiver.
Assume that each connected component of $Q^{s}$ is a simply-laced Dynkin quiver or a simply-laced extended Dynkin quiver.
Then we have $\rho(Q)\leq 2$. 
Furthermore, if all connected components of $Q^{s}$ are only simply-laced Dynkin quivers, then we have $\rho(Q)<2$.
\end{proposition}

\begin{proof}
Let $M=(a_{ij})_{i,j\in Q_{0}}$ be the adjacent matrix of $Q$.
Denote by $\mathcal{C}$ the set of all connected components of $Q^{s}$.
Let $\bm{x}=(x_{i})_{i\in Q_{0}}\in \mathbb{C}^{Q_{0}}$ with $\|\bm{x}\|=1$, where $\|\cdot\|$ is the Euclidean norm.
For each $j^{-}\in Q^{s}_{0}$, let $x_{j^{-}}:=x_{j}$.
Note that $a_{ij}$ equals the number of arrows from $i^+$ to $j^-$ in $Q^s$. Then we have
\begin{align}
\|M\bm{x}\|^{2}
\leq \sum_{i\in Q_{0}}(\sum_{j\in Q_{0}}a_{ij}|x_{j}|)^{2}
=\sum_{\Delta\in \mathcal{C}}\sum_{i^{+}\in \Delta_{0}}(\sum_{j^{-}\in \Delta_{0}}a_{ij}|x_{j^{-}}|)^{2}. \notag
\end{align}
Assume that each connected component of $Q^{s}$ is a simply-laced Dynkin quiver or a simply-laced extended Dynkin quiver.
By Lemma \ref{lem:quad-form}(2) and $\|\bm{x}\|=1$, we obtain
\begin{align}
\|M\bm{x}\|^{2}
\leq \sum_{\Delta\in \mathcal{C}}\sum_{i^{+}\in \Delta_{0}}(\sum_{j^{-}\in \Delta_{0}}a_{ij}|x_{j^{-}}|)^{2}
\leq \sum_{\Delta\in \mathcal{C}}\sum_{j^{-}\in\Delta_{0}^{-}}4|x_{j^{-}}|^{2}= 4. \notag
\end{align}
This implies $\|M\bm{x}\|\leq 2$. 
Since the spectrum radius $\rho(M)$ is the absolute value of a maximal eigenvalue of $M$, we have
\begin{align}
\sup_{\|\bm{x}\|=1}\|M\bm{x}\|\geq \rho(M).\notag
\end{align}
Therefore, $\rho(Q)=\rho(M)\leq 2$ holds.
Furthermore, we assume that all connected components of $Q^{s}$ are only simply-laced Dynkin quivers.
By a similar argument above, it follows from Lemma \ref{lem:quad-form}(1) that $\|M\bm{x}\|^{2}<4$ holds, and hence $\rho(Q)=\rho(M)< 2$.
The proof is complete.
\end{proof}

Now, we are ready to prove Theorem \ref{thm:ub_rf}.

\begin{proof}[Proof of Theorem \ref{thm:ub_rf}]
Let $A$ be a $\tau$-tilting finite algebra.
Let $\mathcal{S}\in \sbrick A$ and $\mathcal{W}_{\mathcal{S}}$ the corresponding wide subcategory by Proposition \ref{prop:ringel}.
Since $A$ is $\tau$-tilting finite, it follows from Propositions \ref{prop:DIRRT-thm412} and \ref{prop:DIRRT-thm418} that there exists a finite-dimensional algebra $B$ such that $\mathcal{W}_{\mathcal{S}}$ is equivalent to $\mod B$ and $Q:=Q_{\mathcal{S}}=Q_{\simp(B)}$.
By \cite[Theorem X.2.4]{ARS95}, there exists a stable equivalence $\underline{\mod}(B/\rad^{2}B)\rightarrow \underline{\mod}\Bbbk Q^{s}$.

Assume that $A$ is of tame representation type.
Since there exist two fully faithful functors $\mod(B/\rad^{2}B)\rightarrow \mod B\rightarrow \mod A$, it follows from \cite[XIX.1.11 Theorem]{SS07} that $B$ and $B/\rad^{2}B$ are of finite representation type or of tame representation type. 
By \cite[Theorem X.2.4]{ARS95} and \cite[Corollary 3.4]{Kr97}, the stable equivalence yields that the path algebra $\Bbbk Q^{s}$ is of finite representation type or of tame representation type. For the former case, the quiver $Q^{s}$ is a disjoint union of simply-laced Dynkin quivers. By Proposition \ref{prop:sep-quiv}, we have $\rho(Q)<2$. On the other hand, for the latter case, the quiver $Q^{s}$ is a disjoint union of simply-laced Dynkin quivers and simply-laced extended Dynkin quivers. Thus we have $\rho(Q)\leq 2$. Hence $\FPdim(A)\leq 2$.
By a similar argument above, if $A$ is of finite representation finite, then we have $\FPdim(A)<2$.
\end{proof}

In the following, we give a proof of Lemma \ref{lem:quad-form}.

\begin{proof}[Proof of Lemma \ref{lem:quad-form}]
Let $\Delta$ be a bipartite quiver whose underlying graph is a simply-laced Dynkin diagram or a simply-laced extended Dynkin diagram.  We set $\Delta_{0}^{-}=\{1,2,\cdots, l\}$ and the elements in $\Delta_{0}^{+}$ is represented as $\bullet$.
If $\Delta$ is of type $\tilde{\mathsf{A}}_1$, then we have $q_{\Delta}(\bm{x})=0$ for all $\bm{x}$.
In the following, we assume that $\Delta$ is not of type $\tilde{\mathsf{A}}_1$.
Then $\Delta$ has no multiple arrows and we have
\begin{align}
q_{\Delta}(\bm{x})=\sum_{j\in \Delta_{0}^{-}}(4-|\mathrm{dp}(j)|)x_{j}^{2}-\sum_{i\in \Delta_{0}^{+}}\sum_{\substack{j, k \in \mathrm{ds}(i)\\j\neq k}} x_{j}x_{k},\notag
\end{align}
where $|\mathrm{dp}(j)|$ is the cardinality of $\mathrm{dp}(j)$.
Indeed, we obtain
\begin{align}
\sum_{i\in \Delta_{0}^{+}}(\sum_{j\in \Delta_{0}^{-}}a_{ij}x_{j})^{2}
&=\sum_{i\in \Delta_{0}^{+}}(\sum_{j\in \mathrm{ds}(i)}x_{j})^2
\notag\\
&=\sum_{i\in \Delta_{0}^{+}}\sum_{j,k\in \mathrm{ds}(i)}x_{j}x_{k}\notag\\
&=\sum_{i\in \Delta_{0}^{+}}\sum_{j\in \mathrm{ds}(i)}x_{j}^{2}+\sum_{i\in \Delta_{0}^{+}}\sum_{\substack{j, k \in \mathrm{ds}(i)\\j\neq k}}x_{j}x_{k}\notag\\
&=\sum_{j\in \Delta_{0}^{-}}|\mathrm{dp}(j)|x_{j}^{2}+\sum_{i\in \Delta_{0}^{+}}\sum_{\substack{j, k \in \mathrm{ds}(i)\\j\neq k}}x_{j}x_{k}.\notag
\end{align}

(1) Assume that the underlying graph of $\Delta$ is a Dynkin diagram of type $\mathsf{A}$, that is, $\Delta$ is one of the following quivers:
\begin{align}
\Delta^{(1)}:\xymatrix@C=4mm{\bullet \ar[r] & 1 & \ar[l] \cdots \ar[r] & l & \ar[l] \bullet}\quad
\Delta^{(2)}:\xymatrix@C=4mm{\bullet \ar[r] & 1 & \ar[l] \cdots \ar[r]&l}\quad
\Delta^{(3)}:\xymatrix@C=4mm{1 & \bullet \ar[l]\ar[r]& \cdots \ar[r]&l}\notag
\end{align}
Then we have the following quadratic forms.
\begin{align}
q_{\Delta^{(1)}}(\bm{x})
&=(x_{1}-x_{2})^{2} +\cdots +(x_{l-1}-x_{l})^{2} +x_{1}^{2}+x_{l}^{2},\notag\\
q_{\Delta^{(2)}}(\bm{x})
&=(x_{1}-x_{2})^{2} +\cdots +(x_{l-1}-x_{l})^{2} +x_{1}^{2}+2x_{l}^{2},\notag\\
q_{\Delta^{(3)}}(\bm{x})
&=(x_{1}-x_{2})^{2} +\cdots +(x_{l-1}-x_{l})^{2} +2x_{1}^{2}+2x_{l}^{2}.\notag
\end{align}

Assume that the underlying graph of $\Delta$ is a Dynkin diagram of type $\mathsf{D}$, that is, $\Delta$ is one of the following quivers:
\begin{align}
\begin{array}{ll}
\Delta^{(4)}:\xymatrix@C=4mm@R=4mm{\bullet \ar[r] & 1 & \ar[l] \cdots \ar[r] & l &\ar[l] \bullet\\& \bullet \ar[u] &&&}&
\Delta^{(5)}:\xymatrix@C=4mm@R=4mm{\bullet \ar[r] & 1 & \ar[l] \cdots \ar[r] & l \\& \bullet \ar[u] &&}\notag\\
\Delta^{(6)}:\xymatrix@C=4mm@R=4mm{1 & \ar[l] \bullet \ar[d] \ar[r] & 3 & \cdots \ar[l] \ar[r] & l & \ar[l] \bullet\\ & 2 &&&&}&
\Delta^{(7)}:\xymatrix@C=4mm@R=4mm{1 & \ar[l] \bullet \ar[d] \ar[r] & 3 & \cdots \ar[l] \ar[r] & l \\ & 2 &&&}
\end{array}\notag
\end{align}
Then we have the following quadratic forms.
\begin{align}
q_{\Delta^{(4)}}(\bm{x})
&=(x_{1}-x_{2})^{2}+\cdots+(x_{l-1}-x_{l})^{2}+x_{l}^{2},\notag\\
q_{\Delta^{(5)}}(\bm{x})
&=(x_{1}-x_{2})^{2}+\cdots+(x_{l-1}-x_{l})^{2}+2x_{l}^{2},\notag\\
q_{\Delta^{(6)}}(\bm{x})
&=(x_{1}-x_{2})^{2}+(\sqrt{2}x_{1}-\frac{1}{\sqrt{2}}x_{3})^{2}+(\sqrt{2}x_{2}-\frac{1}{\sqrt{2}}x_{3})^{2}\notag\\
&\hspace{10mm}+(x_{3}-x_{4})^{2}+\cdots+(x_{l-1}-x_{l})^{2}+x_{l}^{2},\notag\\
q_{\Delta^{(7)}}(\bm{x})
&=(x_{1}-x_{2})^{2}+(\sqrt{2}x_{1}-\frac{1}{\sqrt{2}})^{2}+(\sqrt{2}x_{2}-\frac{1}{\sqrt{2}}x_{3})^{2}\notag\\
&\hspace{10mm}+(x_{3}-x_{4})^{2}+\cdots+(x_{l-1}-x_{l})^{2}+2x_{l}^{2}.\notag
\end{align}

Assume that the underlying graph of $\Delta$ is a Dynkin diagram of type $\mathsf{E}$, that is, $\Delta$ is one of the following quivers:
\begin{align}
\begin{array}{ll}
\Delta^{(8)}:\xymatrix@C=4mm@R=4mm{\bullet \ar[r] & 1 & \ar[l] \bullet \ar[d] \ar[r] & 3 & \ar[l] \bullet\\ && 2 &&}&
\Delta^{(9)}:\xymatrix@C=4mm@R=4mm{1 & \ar[l] \bullet \ar[r] & 2 & \ar[l] \bullet \ar[r] & 3 \\&& \bullet \ar[u] &&}\\
\Delta^{(10)}:\xymatrix@C=4mm@R=4mm{\bullet \ar[r] & 1 & \ar[l] \bullet \ar[d] \ar[r] & 3 & \ar[l] \bullet \ar[r] & 4 \\ && 2 &&&}&
\Delta^{(11)}:\xymatrix@C=4mm@R=4mm{1 & \ar[l] \bullet \ar[r] & 2 & \ar[l] \bullet \ar[r] & 3 & \ar[l] \bullet\\&& \bullet \ar[u] &&&}\\
\Delta^{(12)}:\xymatrix@C=4mm@R=4mm{\bullet \ar[r] & 1 & \ar[l] \bullet \ar[d] \ar[r] & 3 & \ar[l] \bullet \ar[r] & 4 & \ar[l] \bullet\\ && 2 &&&&}&
\Delta^{(13)}:\xymatrix@C=4mm@R=4mm{1 & \ar[l] \bullet \ar[r] & 2 & \ar[l] \bullet \ar[r] & 3 & \ar[l] \bullet\ar[r] & 4\\ && \bullet \ar[u] &&&}
\end{array}\notag
\end{align}
Then we have the following quadratic forms.
\begin{align}
q_{\Delta^{(8)}}(\bm{x})
&=(x_{1}-x_{2})^{2}+(x_{1}-x_{3})^{2}+(x_{2}-x_{3})^{2}+x_{2}^{2},\notag\\
q_{\Delta^{(9)}}(\bm{x})
&=(\sqrt{2}x_{1}-\frac{1}{\sqrt{2}}x_{2})^{2}+(\frac{1}{\sqrt{2}}x_{2}-\sqrt{2}x_{3})^{2}+x_{1}^{2}+x_{3}^{2},\notag\\
q_{\Delta^{(10)}}(\bm{x})
&=(x_{1}-x_{2})^{2}+(x_{1}-x_{3})^{2}\notag\\
&\hspace{10mm}+(\sqrt{2}x_{2}-\frac{1}{\sqrt{2}}x_{3})^{2}+(\frac{1}{\sqrt{2}}x_{3}-\sqrt{2}x_{4})^{2}+x_{4}^{2},\notag\\
q_{\Delta^{(11)}}(\bm{x})
&=(\sqrt{2}x_{1}-\frac{1}{\sqrt{2}}x_{2})^{2}+(\frac{1}{\sqrt{2}}x_{2}-\sqrt{2}x_{3})^{2}+x_{1}^{2},\notag\\
q_{\Delta^{(12)}}(\bm{x})
&=(x_{1}-x_{2})^{2}+(x_{1}-x_{3})^{2}+(\sqrt{2}x_{2}-\frac{1}{\sqrt{2}}x_{3})^{2}+(\frac{1}{\sqrt{2}}x_{3}-\sqrt{2}x_{4})^{2},\notag\\
q_{\Delta^{(13)}}(\bm{x})
&=(\sqrt{3}x_{1}-\frac{1}{\sqrt{3}}x_{2})^{2}+(\sqrt{\frac{2}{3}}x_{2}-\sqrt{\frac{3}{2}}x_{3})^{2}+(\frac{1}{\sqrt{2}}x_{3}-\sqrt{2}x_{4})^{2}+x_{4}^{2}.\notag
\end{align}

Thus we can easily check that all quadratic forms are positive definite.

(2) Assume that the underlying graph of $\Delta$ is an extended Dynkin diagram of type $\tilde{\mathsf{A}}$ or $\tilde{\mathsf{D}}$, that is, $\Delta$ is one of the following quivers:
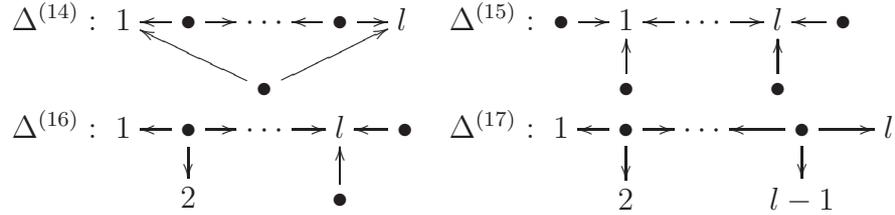
\begin{align}
\begin{array}{ll}
\Delta^{(14)}:\xymatrix@C=4mm@R=4mm{1 & \bullet\ar[l]\ar[r]  & \cdots & \bullet\ar[l]\ar[r] & l\\&&\bullet\ar[ull]\ar[urr]&&}&
\Delta^{(15)}:\xymatrix@C=4mm@R=4mm{ \bullet\ar[r] & 1 &  \cdots\ar[l]\ar[r]  &  l & \bullet\ar[l]\\ & \bullet\ar[u] && \bullet\ar[u] &}\\
\Delta^{(16)}:\xymatrix@C=4mm@R=4mm{1 & \bullet\ar[l]\ar[d]\ar[r] & \cdots \ar[r] &  l & \bullet\ar[l] \\& 2 && \bullet\ar[u] &}&
\Delta^{(17)}:\xymatrix@C=4mm@R=4mm{1 & \bullet\ar[l]\ar[d]\ar[r] & \cdots  &  \bullet\ar[l]\ar[d]\ar[r] & l\\&2&&l-1&}
\end{array}\notag
\end{align}
Then we have the following quadratic forms.
\begin{align}
q_{\Delta^{(14)}}(\bm{x})
&=(x_{1}-x_{2})^{2}+\cdots+(x_{l-1}+x_{l})^{2}+(x_{l}-x_{1})^{2}\notag\\
q_{\Delta^{(15)}}(\bm{x})
&=\left\{\begin{array}{ll} 0 & (\ell=1) \\ (x_{1}-x_{2})^{2}+\cdots+(x_{l-1}-x_{l})^{2} & (\ell \ge 2) \end{array}\right. \notag\\
q_{\Delta^{(16)}}(\bm{x})
&=\left\{\begin{array}{ll}
(x_{1}-x_{2})^{2}+(\sqrt{2}x_{1}-\frac{1}{\sqrt{2}}x_{3})^{2}+(\sqrt{2}x_{2}-\frac{1}{\sqrt{2}}x_{3})^{2} & (l=3)\vspace{1mm}\\
\begin{array}{l}
(x_{1}-x_{2})^{2}+(\sqrt{2}x_{1}-\frac{1}{\sqrt{2}}x_{3})^{2}+(\sqrt{2}x_{2}-\frac{1}{\sqrt{2}}x_{3})^{2}\\
\hspace{10mm}+(x_{3}-x_{4})^{2}+\cdots +(x_{l-1}-x_{l})^{2}
\end{array}
& (l\ge 4 ) 
\end{array}\right.\notag\\
q_{\Delta^{(17)}}(\bm{x})
&=\left\{
\begin{array}{ll}
\begin{array}{l}
(x_1-x_2)^2+(x_1-x_3)^2+(x_1-x_4)^2\\
\hspace{10mm}+(x_2-x_3)^2+(x_2-x_4)^2+(x_3-x_4)^2
\end{array}& (l=4)\vspace{2mm}\\
\begin{array}{l}
(x_{1}-x_{2})^{2}+(\sqrt{2}x_{1}-\frac{1}{\sqrt{2}}x_{3})^{2}+(\sqrt{2}x_{2}-\frac{1}{\sqrt{2}}x_{3})^{2}\\
\hspace{5mm}+(x_{3}-x_{4})^{2}+\cdots+(x_{l-3}-x_{l-2})^{2}\\
\hspace{5mm}+(\frac{1}{\sqrt{2}}x_{l-2}-\sqrt{2}x_{l-1})^{2}+(\frac{1}{\sqrt{2}}x_{l-2}-\sqrt{2}x_{l})^{2}+(x_{l-1}- x_{l})^{2} 
\end{array}& (l\ge 5)
\end{array}\right.
\notag
\end{align}
Note that
\begin{align}
\begin{array}{ll}
q_{\Delta^{(14)}}(1,1,\ldots, 1)=0 & q_{\Delta^{(15)}}(1,1,\ldots, 1)=0\\
q_{\Delta^{(16)}}(1,1,2,2,\ldots, 2,2)=0 & q_{\Delta^{(17)}}(1,1,2,2,\ldots, 2,2,1,1)=0
\end{array}\notag
\end{align}

Assume that the underlying graph of $\Delta$ is an extended Dynkin diagram of type $\tilde{\mathsf{E}}$, that is, $\Delta$ is one of the following quivers:
\begin{align}
\begin{array}{ll}
\Delta^{(18)}:\xymatrix@C=4mm@R=4mm{\bullet\ar[r] & 1 & \bullet\ar[l]\ar[r]\ar[d] & 3 & \bullet\ar[l] \\&& 2 &\bullet\ar[l]&}&
\Delta^{(19)}:\xymatrix@C=4mm@R=4mm{1 & \bullet\ar[l]\ar[r] & 2 & \bullet\ar[l]\ar[r] & 4\\&&\bullet\ar[u]\ar[r] &3&}\\
\Delta^{(20)}:\xymatrix@C=4mm@R=4mm{\bullet\ar[r]& 1 & \bullet\ar[l]\ar[r] & 2 & \bullet\ar[l]\ar[r] & 3 & \bullet\ar[l] \\ &&& \bullet\ar[u] &&&}&
\Delta^{(21)}:\xymatrix@C=4mm@R=4mm{1 & \bullet\ar[l]\ar[r] & 2 & \bullet\ar[l]\ar[r]\ar[d] & 4 & \bullet\ar[l]\ar[r] & 5 \\ &&& 3 &&&}\\
\Delta^{(22)}:\xymatrix@C=4mm@R=4mm{ \bullet\ar[r] & 1 & \bullet\ar[l]\ar[r]\ar[d] & 3 & \bullet\ar[l]\ar[r] & 4 & \bullet\ar[l]\ar[r] & 5 \\ && 2 &&&&&}&
\Delta^{(23)}:\xymatrix@C=4mm@R=4mm{ 1 & \bullet\ar[l]\ar[r] & 2 & \bullet\ar[l]\ar[r] & 3 & \bullet\ar[l]\ar[r] & 4 & \bullet\ar[l] \\ && \bullet\ar[u] &&&&&}\notag
\end{array}
\end{align}
Then we have the following quadratic forms.
\begin{align}
q_{\Delta^{(18)}}(\bm{x})
&=(x_{1}-x_{2})^{2}+(x_{1}-x_{3})^{2}+(x_{2}-x_{3})^{2},\notag\\
q_{\Delta^{(19)}}(\bm{x})
&=(\sqrt{3}x_{1}-\frac{1}{\sqrt{3}}x_{2})^{2}+(\sqrt{3}x_{3}-\frac{1}{\sqrt{3}}x_{2})^{2}+(\sqrt{3}x_{4}-\frac{1}{\sqrt{3}}x_{2})^{2},\notag\\
q_{\Delta^{(20)}}(\bm{x})
&=(\sqrt{2}x_{1}-\frac{1}{\sqrt{2}}x_{2})^{2}+(\frac{1}{\sqrt{2}}x_{2}-\sqrt{2}x_{3})^{2},\notag\\
q_{\Delta^{(21)}}(\bm{x})
&=(\sqrt{3}x_{1}-\frac{1}{\sqrt{3}}x_{2})^{2}+(\frac{\sqrt{2}}{\sqrt{3}}x_{2}-\frac{\sqrt{3}}{\sqrt{2}}x_{3})^{2}+(x_{2}-x_{4})^{2}\notag\\
&\hspace{10mm}+(\frac{\sqrt{3}}{\sqrt{2}}x_{3}-\frac{\sqrt{2}}{\sqrt{3}}x_{4})^{2}+(\frac{1}{\sqrt{3}}x_{4}-\sqrt{3}x_{5})^{2},\notag\\
q_{\Delta^{(22)}}(\bm{x})
&=(\frac{\sqrt{3}}{2}x_{1}-\frac{2}{\sqrt{3}}x_{2})^{2}+(\frac{\sqrt{5}}{2}x_{1}-\frac{2}{\sqrt{5}}x_{3})^{2}+(\frac{\sqrt{5}}{\sqrt{3}}x_{2}-\frac{\sqrt{3}}{\sqrt{5}}x_{3})^{2}\notag\\
&\hspace{10mm}+(\frac{\sqrt{3}}{\sqrt{5}}x_{3}-\frac{\sqrt{5}}{\sqrt{3}}x_{4})^{2}+(\frac{1}{\sqrt{3}}x_{4}-\sqrt{3}x_{5})^{2},\notag\\
q_{\Delta^{(23)}}(\bm{x})
&=(\sqrt{3}x_{1}-\frac{1}{\sqrt{3}}x_{2})^{2}+(\frac{\sqrt{2}}{\sqrt{3}}x_{2}-\frac{\sqrt{3}}{\sqrt{2}}x_{3})^{2}+(\frac{1}{\sqrt{2}}x_{3}-\sqrt{2}x_{4})^{2}.\notag
\end{align}
Note that
\begin{align}
\begin{array}{lll}
q_{\Delta^{(18)}}(1,1,1)=0&q_{\Delta^{(19)}}(1,3,1,1)=0&q_{\Delta^{(20)}}(1,2,1)=0\\
q_{\Delta^{(21)}}(1,3,2,3,1)=0&q_{\Delta^{(22)}}(4,3,5,3,1)=0&q_{\Delta^{(23)}}(1,3,2,1)=0
\end{array}\notag
\end{align}

Thus we can easily check that all quadratic forms are positive semidefinite but not positive definite.
The proof is complete.
\end{proof}

\section{Frobenius--Perron dimension of Nakayama algebras}

In this section, we determine the Frobenius--Perron dimension of a Nakayama algebra.
Namely, the aim is to prove the following theorem.

\begin{theorem}\label{thm:FPdim_of_Nakayama}
Let $A$ be a Nakayama algebra.
Then the following statements hold.
\begin{itemize}
\item[(1)] If $A$ is linear, then we have $\FPdim(A)=0$.
\item[(2)] If $A$ is cyclic, then we have $\FPdim(A)=1$.
\end{itemize}
\end{theorem}

To prove the theorem above, we first recall the definition and basic properties of Nakayama algebras. For details, see \cite[Chapter V]{ASS06}.
We call an algebra $A$ a \emph{Nakayama algebra} if $A$ is isomorphic to the bound quiver algebra $\Bbbk Q/I$, where each connected component of $Q$ is isomorphic to one of the following two quivers
\begin{align}
\mathsf{A}_{n}^{\leftarrow}: \xymatrix{1&2\ar[l]_-{\alpha_{1}}&\cdots\ar[l]_-{\alpha_{2}}&n-1\ar[l]_-{\alpha_{n-2}}&n\ar[l]_-{\alpha_{n-1}}},\notag\\
\tilde{\mathsf{A}}_{n}^{\leftarrow}: \xymatrix{1\ar@/_5mm/[rrrr]_-{\alpha_{n}}&2\ar[l]_-{\alpha_{1}}&\cdots\ar[l]_-{\alpha_{2}}&n-1\ar[l]_-{\alpha_{n-2}}&n\ar[l]_-{\alpha_{n-1}}},\notag
\end{align}
and $I$ is an admissible ideal of $\Bbbk Q$.
Note that $\tilde{\mathsf{A}}_{1}^{\leftarrow}$ is a quiver with one vertex and one loop.
We call $\mathsf{A}_{n}^{\leftarrow}$ a \emph{linear quiver} and $\tilde{\mathsf{A}}_{n}^{\leftarrow}$ a \emph{cyclic quiver}.
The (connected) Nakayama algebra $A$ is said to be \emph{linear} (respectively, \emph{cyclic}) if $Q$ is isomorphic to $\mathsf{A}_{n}^{\leftarrow}$ (respectively, $\tilde{\mathsf{A}}_{n}^{\leftarrow}$) for some $n\geq 1$.
It is well known that all Nakayama algebras are representation-finite (see \cite[V.3.5 Theorem]{ASS06}), and hence $\tau$-tilting finite.

Next, we give the description of indecomposable modules over a Nakayama algebra.
Let $A$ be a connected Nakayama algebra with $n$ non-isomorphic simple modules. 
Since each indecomposable $A$-module $M$ is uniserial (see \cite[V.3.2. Theorem]{ASS06}), it has a unique composition series:
\begin{align}
0=M_{0}\subset M_{1}\subset M_{2}\subset \cdots \subset M_{l}=M.\notag
\end{align}
Thus $M$ is uniquely determined by its (simple) socle $M_{1}=S(i)$ for some $i\in Q_{0}$ and its length $l=\ell(M)$.
We write such a module as $M(i;l)$.
By the definition of the Auslander--Reiten translations $\tau$, we have $\tau M(i;l)=M(i-1;l)$ if $M(i;l)$ is non-projective.
If $M(i;l)$ is projective, then we put $\tau M(i;l):=0$.
For a cyclic Nakayama algebra, we identify $i\in Q_{0}$ with $i\pm n$, e.g., $M(n+1;l)=M(1;l)$.
All bricks and indecomposable $\tau$-rigid modules are characterized by the following conditions.

\begin{lemma}\label{lem:brick-tau-rigid}
Let $A$ be a connected Nakayama algebra with $n$ non-isomorphic simple modules and let $M$ be an indecomposable $A$-module.
Then the following statements hold.
\begin{itemize}
\item[(1)] $M$ is a brick if and only if $\ell(M)\leq n$.
\item[(2)] $M$ is $\tau$-rigid if and only if $M$ is projective or $\ell(M)<n$.
\end{itemize}
\end{lemma}

\begin{proof}
By the definition, the statement (1) is clear. The statement (2) follows from \cite[Proposition 2.5]{Ad16}.
\end{proof}

In the following, we give a proof of Theorem \ref{thm:FPdim_of_Nakayama}.
First, we show Theorem \ref{thm:FPdim_of_Nakayama}(1).

\begin{proof}[Proof of Theorem \ref{thm:FPdim_of_Nakayama}(1)]
Let $A$ be a linear Nakayama algebra, that is, a factor algebra of the path algebra $\Bbbk\mathsf{A}_{n}^{\leftarrow}$ for some $n\geq 1$. 
Since $\Bbbk\mathsf{A}_{n}^{\leftarrow}$ is clearly representation-directed, it follows from Proposition \ref{prop:cc23-cor4.3}(3) that $\FPdim(A)=0$ holds.
\end{proof}

We show Theorem \ref{thm:FPdim_of_Nakayama}(2).
Assume that $A$ is a cyclic Nakayama algebra.
To calculate the Frobenius--Perron dimension of $A$, it is necessary to determine the shapes of the Ext-quivers of semibricks in $\mod A$.
We observe the Bongartz completion for a $\tau$-rigid pair $(M,0):=(M(i;l),0)$.
If $M$ is projective, then $(A,0)$ is the Bongartz completion of $(M,0)$.
Assume that $M$ is not projective.
Since $A$ is cyclic, we may assume that $i=1$.
By Lemma \ref{lem:brick-tau-rigid}(2), we have $l<\min\{ \ell(P(l)), n\}$.
Define an $A$-module $\widetilde{M}$ as $\widetilde{M}=M\oplus X\oplus P$, where
\begin{align}
X:=\bigoplus_{1\leq j<l}M(1;j)\text{ and }P:=\bigoplus_{l\leq k<n}P(k). \notag
\end{align}
Note that $\End_{A}(X)$ is an isomorphic to the path algebra $\Bbbk \mathsf{A}_{l-1}^{\leftarrow}$ and $\End_{A}(P)$ is a Nakayama algebra.
The module $\widetilde{M}$ has the following properties.

\begin{lemma}\label{lem:end-bon-com}
The following statements hold.
\begin{itemize}
\item[(1)] $(\widetilde{M}, 0)$ is the Bongartz completion of $(M,0)$.
\item[(2)] $\End_{A}(\widetilde{M})/[M]$ is a Nakayama algebra.
\end{itemize}
\end{lemma}

\begin{proof}
(1) First, we show that $(\widetilde{M},0)$ is a $\tau$-tilting pair. 
Since $|\widetilde{M}|=n$ holds, it is enough to claim that $\widetilde{M}$ is $\tau$-rigid.
By a property of the Auslander--Reiten translation $\tau$, we obtain 
\begin{align}
\displaystyle \tau\widetilde{M}=M(n;l)\oplus \bigoplus_{1\leq j <l}M(n;j).\notag
\end{align}
Since $M(1;l)$ and $M(1;j)$ ($1\leq j<l$) do not contain $S(n)$ as a composition factor, we have $\Hom_{A}(\widetilde{M},\tau\widetilde{M})=0$.

To prove that $(\widetilde{M},0)$ is the Bongartz completion of $(M,0)$, it is enough to show that, for each $M(1;j)$ ($1\leq j<l$), there exists no exact sequence 
\begin{align}
L\rightarrow M'\xrightarrow{g} M(1;j)\rightarrow 0 \notag
\end{align}
such that $M'\in \add(\widetilde{M}/M(1;j))$.
If such an exact sequence exists, then $g$ is surjective. 
However, there exist no surjective maps from each module in $\add(\widetilde{M}/M(1;j))$ to $M(1;j)$.
This is a contradiction. 
Thus there exists no $\tau$-tilting pairs $(M',0)$ such that $M\in \add M'$ and $(\widetilde{M},0)<(M',0)$. 
Thus $(\widetilde{M},0)$ is the Bongartz completion of $(M,0)$.

(2) By $\widetilde{M}=M\oplus X\oplus P$, we have an algebra isomorphism
\begin{align}\label{seq:matrix-alg}
\End_{A}(\widetilde{M})\cong\begin{bmatrix}
\Hom_{A}(M,M)&\Hom_{A}(X,M)&\Hom_{A}(P,M)\\
\Hom_{A}(M,X)&\Hom_{A}(X,X)&\Hom_{A}(P,X)\\
\Hom_{A}(M,P)&\Hom_{A}(X,P)&\Hom_{A}(P,P)
\end{bmatrix}
\end{align}
Since $X$ does not have $\top(P)$ as a composition factor, $\Hom_{A}(P,X)=0$ holds.
Furthermore, a non-zero morphism in $\Hom_{A}(M(1;j),P)$ factors through $M(1;\ell)$. Thus $\Hom_{A}(X,P)/[M](X,P)=0$ holds, where $[M](X,P)$ is a subspace of $\Hom_{A}(X,P)$ consisting of morphisms factoring through some module in $\add M$.
By \eqref{seq:matrix-alg}, we have an algebra isomorphism 
\begin{align}
\End_{A}(\widetilde{M})/[M]\cong (\End_{A}(X)/[M](X,X))\times(\End_{A}(P)/[M](P,P)).\notag
\end{align}
Since the class of Nakayama algebras is closed under taking factor algebras, we have the assertion.
\end{proof}

The following proposition tells us the shape of the Ext-quiver of a semibrick.

\begin{proposition}\label{prop:nakayama-key}
Let $\mathcal{S}(\neq \emptyset)$ be a semibrick in $\mod A$.
Then each connected component of $Q_{\mathcal{S}}$ is isomorphic to a linear quiver or a cyclic quiver. In particular, we have $\rho(Q_{\mathcal{S}})\leq 1$.
Furthermore, $\rho(Q_{\mathcal{S}})=1$ if and only if $Q_{\mathcal{S}}$ contains a cyclic quiver.
\end{proposition}

\begin{proof}
Let $\mathcal{S}$ be a semibrick in $\mod A$ and let $\mathcal{W}$ be the corresponding wide subcategory of $\mod A$ by Proposition \ref{prop:ringel}. 
Note that $\mathcal{S}=\simp(\mathcal{W})$ and $Q_{\mathcal{S}}=Q_{\simp(\mathcal{W})}$.
Since $A$ is $\tau$-tilting finite, it follows from Proposition \ref{prop:DIRRT-thm418} that $\mathcal{W}$ is a $\tau$-perpendicular category, that is, there exists a $\tau$-rigid pair $(N,P)$ for $\mod A$ such that $\mathcal{W}=\mathcal{W}(N,P)$.

We show the assertion by induction on $n:=|A|$.
If $n=1$ holds, then we have $\wide A=\{ \mod A, 0\}$. Thus the assertion clearly holds.
Assume that $n\neq 1$.
If $N=0$ holds, then $\mathcal{W}$ is equivalent to $\mod (A/AeA)$, where $\add eA=\add P$.
Since $A/AeA$ is a Nakayama algebra, we have the assertion.
If $N\neq 0$ holds, then we can take an indecomposable direct summand $M$ of $N$.
Let $(M^{+},0)$ be the Bongartz completion of $(M,0)$.
Then $B:=\End_{A}(M^{+})/[M]$ is a Nakayama algebra and $|B|=|A|-1$.
Indeed, if $M$ is projective, then we have $M^{+}=A$ and $\End_{A}(M^{+})/[M]\cong A/A\varepsilon A$ for some idempotent $\varepsilon\in A$. 
On the other hand, if $M$ is non-projective, then it follows from Lemma \ref{lem:end-bon-com} that $B$ is a Nakayama algebra.
Furthermore, by Proposition \ref{prop:DIRRT-thm412}, we have an equivalence of (abelian) categories 
\begin{align}
F: \mathcal{W}(M,0)\rightarrow \mod B.\notag
\end{align}
This implies that $\mathcal{W}':=F(\mathcal{W})$ is a wide subcategory of $\mod B$ and $\mathcal{W}\cong\mathcal{W}'$. In particular, we have a quiver isomorphism $Q_{\simp(\mathcal{W})}\cong Q_{\simp(\mathcal{W}')}$.
By the induction hypothesis, each connected component of $Q_{\simp(\mathcal{W}')}$ is isomorphic to $\mathsf{A}_{n}^{\leftarrow}$ or $\tilde{\mathsf{A}}_{n}^{\leftarrow}$ for some $n\geq 1$.
The proof is complete.
\end{proof}

Now, we are ready to prove Theorem \ref{thm:FPdim_of_Nakayama}(2).

\begin{proof}[Proof of Theorem \ref{thm:FPdim_of_Nakayama}(2)]
Assume that the Gabriel quiver of $A$ is isomorphic to $\tilde{\mathsf{A}}_{n}^{\leftarrow}$ for some $n\geq 1$.
By Proposition \ref{prop:nakayama-key}, we have
\begin{align}
1=\rho(\tilde{\mathsf{A}}_{n}^{\leftarrow})\leq \FPdim(A):=\sup\{ \rho(Q_{\mathcal{S}})\mid \mathcal{S}\in \sbrick(A)\}\leq 1. \notag
\end{align}
The proof is complete.
\end{proof}

\section{Frobenius--Perron dimension of generalized preprojective algebras of Dynkin type}

In this section, we study the Frobenius--Perron dimension of a generalized preprojective algebra of Dynkin type in the sense of Geiss--Leclerc--Schr\"{o}er \cite{GLS17}.

\subsection{Cartan matrices and Coxeter groups}

In this subsection, we recall the definitions of  Cartan matrices and Coxeter groups associated to Dynkin diagrams:

\begin{align}
\begin{array}{ll}
\mathsf{A}_{n}: \xymatrix@C=4mm{1 \ar@{-}[r]& 2 \ar@{-}[r] & \cdots \ar@{-}[r]& n-1 \ar@{-}[r]& n}&
\mathsf{E}_{6}: \xymatrix@C=4mm@R=4mm{1 \ar@{-}[r]& 2 \ar@{-}[r] & 3 \ar@{-}[d]\ar@{-}[r]& 5  \ar@{-}[r]& 6\\&&4&&}\\
\mathsf{B}_{n}: \xymatrix@C=4mm{1 \ar@{-}[r]& 2 \ar@{-}[r] & \cdots \ar@{-}[r]& n-1 \ar@{=}|(0.7)@{>}[r]& n}&
\mathsf{E}_{7}: \xymatrix@C=4mm@R=4mm{1 \ar@{-}[r]& 2 \ar@{-}[r] & 3 \ar@{-}[d]\ar@{-}[r]& 5  \ar@{-}[r]& 6 \ar@{-}[r]& 7\\&&4&&}\\
\mathsf{C}_{n}: \xymatrix@C=4mm{1 \ar@{-}[r]& 2 \ar@{-}[r] & \cdots \ar@{-}[r]& n-1 &\ar@{=}|(0.4)@{>}[l] n}&
\mathsf{E}_{8}: \xymatrix@C=4mm@R=4mm{1 \ar@{-}[r]& 2 \ar@{-}[r] & 3 \ar@{-}[d]\ar@{-}[r]& 5  \ar@{-}[r]& 6 \ar@{-}[r]& 7 \ar@{-}[r]& 8\\&&4&&}\\
\mathsf{D}_{n}: \xymatrix@C=4mm@R=4mm{1 \ar@{-}[r]& 2 \ar@{-}[r] & \cdots \ar@{-}[r]& n-2 \ar@{-}[d] \ar@{-}[r]& n\\&&&n-1&}&
\mathsf{F}_{4}: \xymatrix@C=4mm{1 \ar@{-}[r]& 2 \ar@{=}|(0.6)@{>}[r]& 3\ar@{-}[r] & 4}\hspace{5mm}
\mathsf{G}_{2}: \xymatrix@C=4mm{1 \ar@3{-}|(0.6)@{>}[r]& 2}
\end{array}\notag
\end{align}
Dynkin diagrams $\mathsf{A}_{n}$, $\mathsf{D}_{n}$, $\mathsf{E}_{6}$, $\mathsf{E}_{7}$, and $\mathsf{E}_{8}$ are called \emph{simply-laced} Dynkin diagrams.

We start with recalling the definition of Cartan matrices of Dynkin type.
Let $\mathsf{X}_{n}$ be a Dynkin diagram with $n$ vertices.
We define an $n\times n$-matrix $C(\mathsf{X}_{n}):=(c_{ij})$, called the \emph{Cartan matrix} of $\mathsf{X}_{n}$, as follows:
\begin{itemize}
\item if $i=j$, then $c_{ii}=2$,
\item if $i\neq j$, then $(c_{ij},c_{ji})=
\begin{cases}
\ (-1,-1) &\text{if $\xymatrix{i\ar@{-}[r]&j}$},\\
\ (-1,-2) &\text{if $\xymatrix{i\ar@{=}|(0.6)@{>}[r]&j}$},\\
\ (-1,-3) &\text{if $\xymatrix{i\ar@3{-}|(0.6)@{>}[r]&j}$},\\
\ (0,0) &\text{otherwise.}
\end{cases}$
\end{itemize}
It is known that a Cartan matrix $C:=C(\mathsf{X}_{n})$ has a \emph{symmetrizer}, that is, there exists a diagonal matrix $D=\diag(c_{1},c_{2},\ldots,c_{n})$ such that $c_{1}, c_{2},\ldots, c_{n}\in \mathbb{Z}_{\geq 1}$ and $DC$ is symmetric.
Although a symmetrizer $D$ is not unique, it is written as 
\begin{align}
D=
\begin{cases}
\ c\diag(1,\ldots,1) & \text{if $\mathsf{X}_{n}=\mathsf{A}_{n},\mathsf{D}_{n},\mathsf{E}_{n}$},\\
\ c\diag(2,\dots,2,1) & \text{if $\mathsf{X}_{n}=\mathsf{B}_{n}$},\\
\ c\diag(1,\dots,1,2) & \text{if $\mathsf{X}_{n}=\mathsf{C}_{n}$},\\
\ c\diag(2,2,1,1) & \text{if $\mathsf{X}_{n}=\mathsf{F}_{n=4}$},\\
\ c\diag(3,1) & \text{if $\mathsf{X}_{n}=\mathsf{G}_{n=2}$},
\notag
\end{cases}
\end{align}
where $c$ is a positive integer. We say that a symmetrizer $D$ is \emph{minimal} if $c=1$.

Next, we recall the definitions of a Coxter group and its right weak order.
For details, refer to \cite{BB05}.
Let $C$ be the Cartan matrix of a Dynkin diagram $\mathsf{X}_{n}$.
The Coxeter group $W=W(C)$ associated to $C$ is defined by generators $s_{1},s_{2},\ldots,s_{n}$ and relations $(s_{i}s_{j})^{m_{ij}}=1$, where 
\begin{align}
m_{ij}=\begin{cases}
1 &\text{if $\xymatrix{i=j}$},\\
2 &\text{if $\xymatrix{i&j}$},\\
3 &\text{if $\xymatrix{i\ar@{-}[r]&j}$},\\
4 &\text{if $\xymatrix{i\ar@{=}|(0.6)@{>}[r]&j}$},\\
6 &\text{if $\xymatrix{i\ar@3{-}|(0.6)@{>}[r]&j}$}.
\end{cases}\notag
\end{align}
Each element $w\in W$ can be written in the form $w=s_{i_{1}}s_{i_{2}}\cdots s_{i_{\ell}}$.
If $\ell$ is minimum, then it is called the \emph{length} of $w$ and denoted by $l(w)$.
In this case, an expression $s_{i_{1}}s_{i_{2}}\cdots s_{i_{\ell}}$ of $w$ is said to be  \emph{reduced} of $w$. Note that a (reduced) expression of $w$ is not necessarily unique. 

For elements $u,w\in W$, we write $u\leq_{R} w$ if there exist $s_{i_1},\dots,s_{i_k}\in W$ such that
\begin{align}
w=us_{i_{1}}\cdots s_{i_{k}}\text{ and }l(w)=l(u)+k.\notag
\end{align}
It is obvious that $\leq_{R}$ gives a partial order on $W$.
We call this partial order the \emph{right weak order} on $W$. 
It is known that $(W,\le_{R})$ forms a finite lattice (for example, see \cite[Section\;3.2]{BB05}).
By definition, the minimum element of $(W,\leq_{R})$ is the identity $1\in W$.
Furthermore, since $(W,\le_{R})$ is a finite lattice, the maximum element (i.e., the longest element) $w_{0}\in W$ exists. 
For a non-empty subset $J$ of $\{s_{1},s_{2},\ldots, s_{n}\}$, we set $w_{0}(J):=\vee J$.

\begin{lemma}[{\cite[Proposition 3.1.6 and Lemma 3.2.4]{BB05}}]\label{lem:coxeter-sub}
Let $J$ be a non-empty subset of $\{s_{1},s_{2},\ldots, s_{n}\}$.
Then the following statements hold.
\begin{itemize}
\item[(1)] If $u\leq_{R} w$ holds, then we have $[u,w]\simeq [1,u^{-1}w]$.
\item[(2)] If $w\leq_R w s_{j}$ holds for each $s_{j}\in J$, then we have
\begin{align}
\vee\{ws_j\mid s_{j}\in J\}=ww_{0}(J).\notag
\end{align}
\end{itemize}
\end{lemma}

\subsection{Generalized preprojective algebras}

In this subsection, we recall the definition of generalized preprojective algebras of Dynkin type. For details, refer to \cite{GLS17}.
Let $\mathsf{X}_{n}$ be a Dynkin diagram and $C=(c_{ij})$ the Cartan matrix of $\mathsf{X}_{n}$ with symmetrizer $D=\diag(c_{1}, \ldots, c_{n})$.
Note that, if $c_{ij}<0$, then $\gcd(|c_{ij}|,|c_{ji}|)=1$.
Fix an acyclic quiver $\Delta$ satisfying the condition that $c_{ij}\neq 0$ if and only if there exists an edge (in the underlying graph of $\Delta$) between $i$ and $j$.
We set 
\begin{align}
\Omega:=\{(i,j)\mid \text{there exists an arrow from $i$ to $j$ in $\Delta$}\},\notag
\end{align}
$\Omega^{\ast}:=\{ (i,j)\mid (j,i)\in \Omega\}$ and $\overline{\Omega}:=\Omega\sqcup\Omega^{\ast}$.

Define a quiver $Q=(Q_{0},Q_{1})$ as 
\begin{align}
&Q_{0}:=\{ 1,2,\ldots, n\},\notag\\
&Q_{1}:=\{ a_{ij}:i\rightarrow j\mid (i,j)\in \overline{\Omega}\}\sqcup\{ \epsilon_{i}:i\rightarrow i \mid i\in Q_{0}\}, \notag
\end{align}
and elements $\rho_{1},\rho_{2},\rho_{3}\in \Bbbk Q$ as
\begin{align}
&\rho_{1} = \underset{i=1}{\overset{n}{\sum}} \epsilon_i^{c_i},\notag\\
&\rho_{2} = \underset{(i,j)\in \overline{\Omega}}{\sum} \left(\epsilon_i^{|c_{ji}|}a_{ij}-a_{ij}\epsilon_j^{|c_{ij}|}\right),\notag\\
&\rho_{3} = \sum_{(j,i)\in\overline{\Omega}}\sum_{f=0}^{|c_{ji}|-1}\operatorname{sgn}(i,j)\epsilon_{i}^{f}a_{ij}a_{ji}\epsilon_{i}^{|c_{ji}|-1-f},\notag
\end{align}
where $\operatorname{sgn}(i,j)=1$ if $(i,j)\in \Omega$ and $\operatorname{sgn}(i,j)=-1$ if $(i,j)\in \Omega^{\ast}$.
Then the bound quiver algebra $\Pi(C,D):=\Bbbk Q/\langle \rho_{1},\rho_{2},\rho_{3}\rangle$ does not depend on the choice of $\Delta$ (up to isomorphisms).
We call $\Pi(C,D)$ the \emph{generalized preprojective algebra} associated with $(C,D)$.
By \cite[Theorem 1.7]{GLS17}, the algebra $\Pi(C,D)$ is finite-dimensional.

\begin{remark}
In general, generalized preprojective algebras can be defined for symmetrizable generalized Cartan matrices. By \cite[Theorem 1.7]{GLS17}, the generalized preprojective algebra of a symmetrizable generalized Cartan matrix $C$ is finite-dimensional if and only if $C$ is of Dynkin type. In this paper, we focus on finite-dimensional algebras and, consequently, restrict our attention to generalized preprojective algebras of Dynkin type.
\end{remark}

We give a relationship between $\tau$-tilting pairs for $\Pi(C,D)$ and the Coxeter group associated to $C$. For each $i\in Q_{0}$, let $I_{i}:=A(1-e_{i})A$. For a reduced expression $w=s_{i_{1}}s_{i_{2}}\cdots s_{i_{\ell}}$, we put $I_{w}:=I_{i_{1}}I_{i_{2}}\cdots I_{i_{\ell}}$. 

\begin{proposition}[{\cite[Theorem 2.30]{Miz14}}, {\cite[Theorem 5.16, Theorem 5.17]{FG19}}]\label{prop:miz}
Let $\Pi= \Pi(C,D)$ be the generalized preprojective algebra and let $W=W(C)$ be the Coxeter group.
Then the assignment $w\mapsto I_{w}$ induces a poset isomorphism
\begin{align}
(W,\le_R^{\mathrm{op}})\rightarrow (\sttilt(\Pi^{\mathrm{op}}), \leq), \notag
\end{align}
where $\le_{R}^{\mathrm{op}}$ is the opposite order of the right weak order on $W$.
In particular, $\Pi$ and $\Pi^{\mathrm{op}}$ are $\tau$-tilting finite.
\end{proposition}

The proposition above says that $I_{w}$ can be uniquely extended to a $\tau$-tilting pair $\underline{I_{w}}:=(I_{w}, P)$ for some projective module $P$.
For simplicity, we identify $\underline{I_{w}}$ with $I_{w}$.

\subsection{The Frobenius--Perron dimension}
The aim of this subsection is to show the following theorem.

\begin{theorem}\label{thm:FPdim-preproj}
Let $C$ be a Cartan matrix of Dynkin type with a symmetrizer $D$ and let $A:= \Pi(C,D):=\Bbbk Q/\langle \rho_{1},\rho_{2},\rho_{3} \rangle$ be the generalized preprojective algebra associated with $(C,D)$.
Then $\FPdim(A)=\rho(Q)$ holds. Furthermore, the spectral radius $\rho(Q)$ is given by the tables in Theorem \ref{mainthm:FPdim_NA_PPA}.
\end{theorem}

To prove Theorem \ref{thm:FPdim-preproj}, the following proposition plays an important role.

\begin{proposition}\label{prop:mur-2424}
Let $A:=\Pi(C,D)^{\mathrm{op}}$.
Let $\mathcal{W}$ be a wide subcategory of $\mod A$.
Then there exists an idempotent $e\in A$ such that $\mathcal{W}$ is equivalent to $\mod (A/\langle e\rangle)$. 
\end{proposition}

\begin{proof}
By Proposition \ref{prop:miz}, $A$ is $\tau$-tilting finite.
Let $\mathcal{W}$ be a wide subcategory of $\mod A$.
By Proposition \ref{prop:DIRRT-thm418}, there exists a $\tau$-rigid pair $X$ such that $\mathcal{W}=\mathcal{W}(X)$.
Let $W$ be the Coxeter group associated to $C$. 
By Proposition \ref{prop:miz}, there exist $w,w'\in W$ such that $X^{+}=I_{w}$ and $X^{-}=I_{w'}$, where $X^{+}$ is the Bongartz completion and $X^{-}$ is the co-Bongartz completion.
Let $J=\{ s_{i_{1}}, \ldots, s_{i_{\ell}}\}:=\{ s_{i}\mid I_{ws_{i}}\in [I_{w'},I_{w}]\}$.
Then $X$ coincides with a maximal common direct summand of $I_{w}, I_{ws_{i_{1}}},\ldots, I_{ws_{i_{l}}}$. 
By Proposition \ref{prop:kase35min}(2), we have 
\begin{align}
I_{w'}=\wedge\{ I_{w}, I_{ws_{i_{1}}},\ldots, I_{ws_{i_{l}}} \}.\notag
\end{align}
The poset isomorphism in Proposition \ref{prop:miz} yields $w'=\vee \{ w, ws_{i_{1}}, \ldots, ws_{i_{l}}\}$.
By Lemma \ref{lem:coxeter-sub}(2), we have $w'=ww_{0}(J)$, where $w_{0}(J):=\vee J$. 
Thus we obtain
\begin{align}
I_{w_{0}(J)}=\wedge\{ I_{1}, I_{s_{i_{1}}}, \ldots, I_{s_{i_{l}}}\}. \notag
\end{align}
Since a maximal common direct summand of $I_{1}, I_{s_{i_{1}}}, \ldots, I_{s_{i_{l}}}$ is given by the form $(P,0)$, where $P$ is a projective $A$-module, it follows from Proposition \ref{prop:kase35min}(2) that 
\begin{align}
\mathrm{Int}(P,0)=[I_{w_{0}(J)}, I_{1}].\notag
\end{align}
In particular, we obtain
\begin{align}
\mathcal{W}(P,0)=\Fac I_{1} \cap I_{w_{0}(J)}^{\perp}=\mod A \cap I_{w_{0}(J)}^{\perp}=I_{w_{0}(J)}^{\perp}.\notag
\end{align}
We can check that the Bongartz completion of $(P,0)$ coincides with $(A,0)$.
By Proposition \ref{prop:DIRRT-thm412}, we have equivalences of categories
\begin{align}
I_{w_{0}(J)}^{\perp}=\mathcal{W}(P,0)\cong \mod(\End_{A}(A)/[P])\cong \mod(A/\langle e\rangle),\notag
\end{align}
where $e\in A$ is the idempotent corresponding to $P$.
On the other hand, by \cite[Proposition 2.42 and Proposition 2.44]{Mur22}, we have an equivalence of categories
\begin{align}
I^{\perp}_{w_{0}(J)}\rightarrow \Fac I_{w}\cap I_{w'}^{\perp}=\mathcal{W}(X)=\mathcal{W}.\notag
\end{align}
The proof is complete.
\end{proof}

Now, we are ready to prove Theorem \ref{thm:FPdim-preproj}.

\begin{proof}[Proof of Theorem \ref{thm:FPdim-preproj}]
Let $Q$ be the Gabriel quiver of $A$.
Since $Q$ is symmetric, it is also the Gabriel quiver of the opposite algebra $A^{\mathrm{op}}$.
Let $\mathcal{S}(\neq \emptyset)$ be a semibrick in $\mod A^{\mathrm{op}}$ and $\mathcal{W}$ the corresponding wide subcategory by Proposition \ref{prop:ringel}.
By Proposition \ref{prop:mur-2424}, there exists an idempotent $e\in A^{\mathrm{op}}$ such that $\mathcal{W}\cong \mod(A^{\mathrm{op}}/\langle e\rangle)$. 
Since the Gabriel quiver of $A^{\mathrm{op}}/\langle e\rangle$ is a subquiver of $Q$, it follows from Lemma \ref{lem:spectrad-subquiver}(1) that the inequality $\rho(Q_{\mathcal{S}})\leq \rho(Q)$ holds. This implies that $\FPdim(A^{\mathrm{op}})=\rho(Q)$.
By the standard $\Bbbk$-duality, we have $\FPdim(A)=\FPdim(A^{\mathrm{op}})=\rho(Q)$.
The remaining assertion follows from the next proposition.
\end{proof}

\begin{proposition}\label{prop:spect-rad-dynkin}
Consider the Gabriel quiver $Q$ of a generalized preprojective algebra of Dynkin type.
Then the spectral radius $\rho(Q)$ is given by Table 1 and Table 2 in Theorem \ref{mainthm:FPdim_NA_PPA}. 
\end{proposition}

\begin{proof}
Assume that $\mathsf{X}_{n}$ is a simply-laced Dynkin diagram.
Then the adjacency matrix of $Q^{\circ}$ is that of $\mathsf{X}_{n}$.
The spectral radius $\rho(\mathsf{X}_{n})$ of (the adjacency matrix of) $\mathsf{X}_{n}$ is given by \cite{DFGKK13}. 
If $D$ is minimal, then $Q=Q^{\circ}$ holds. Thus we have $\rho(Q)=\rho(\mathsf{X}_{n})$.
If $D$ is not minimal, then $Q$ is obtained from $Q^{\circ}$ adding one loop in each vertex. Thus we have $M(Q)=M(Q^{\circ})+E_{n}$, where $E_{n}$ is an identity matrix. 
This implies that $\rho(Q)=\rho(\mathsf{X}_{n})+1$ holds.

Assume that $\mathsf{X}_{n}$ is a non-simply-laced Dynkin diagram.
If $D$ is not minimal, then $Q$ is isomorphic to the Gabriel quiver of the generalized preprojective algebra of Dynkin type $\mathsf{A}_{n}$. 
Thus we have $M(Q)=M(\mathsf{A}_{n})+E_{n}$, and hence $\rho(Q)=\rho(\mathsf{A}_{n})+1$
In the following, we assume that $D$ is minimal.
Then the Gabriel quivers are given by the following list.

\begin{align}
\begin{array}{ll}
\xymatrix{1\ar@<1mm>[r]\ar@(ul,ur)&2\ar@<1mm>[l]\ar@(ul,ur)\ar@<1mm>[r]&\cdots \ar@<1mm>[l]\ar@<1mm>[r]&n-1\ar@<1mm>[l]\ar@(ul,ur)\ar@<1mm>[r]&n\ar@<1mm>[l]}&\text{if $\mathsf{X}_{n}=\mathsf{B}_{n}$},\\[6mm]
\xymatrix{1\ar@<1mm>[r]&2\ar@<1mm>[l]\ar@<1mm>[r]&\cdots \ar@<1mm>[l]\ar@<1mm>[r]&n-1\ar@<1mm>[l]\ar@<1mm>[r]&n\ar@<1mm>[l]\ar@(ul,ur)}&\text{if $\mathsf{X}_{n}=\mathsf{C}_{n}$},\\[6mm]
\xymatrix{1\ar@<1mm>[r]\ar@(ul,ur)&2\ar@<1mm>[l]\ar@(ul,ur)\ar@<1mm>[r]&3\ar@<1mm>[l]\ar@<1mm>[r]&4\ar@<1mm>[l]}&\text{if $\mathsf{X}_{n}=\mathsf{F}_{n=4}$},\\[6mm]
\xymatrix{1\ar@(ul,ur)\ar@<1mm>[r]&2\ar@<1mm>[l]}&\text{if $\mathsf{X}_{n}=\mathsf{G}_{n=2}$}.
\end{array}\notag
\end{align}
Note that the quiver of $\mathsf{B}_{1}$ has no loops and that of $\mathsf{C}_{1}$ has exactly one loop.
For $\mathsf{X}_{n}\in \{ \mathsf{F}_{n=4}, \mathsf{G}_{n=2}\}$, we can easily check that $\rho(Q_{A})=\frac{1+\sqrt{13}}{2}$ if $\mathsf{X}_{n}=\mathsf{F}_{n=4}$ and $\frac{1+\sqrt{5}}{2}$ if $\mathsf{X}_{n}=\mathsf{G}_{n=2}$.
Assume that $\mathsf{X}_{n}=\mathsf{B}_{n}$.
Let $f_{n}(x)$ be the characteristic polynomial of the adjacency matrix of $Q$.
Then we have a recurrence relation
\begin{align}\label{seq:rr-b1}
f_{n+1}(x)=(x-1)f_{n}(x)-f_{n-1}(x).
\end{align}
Let $x_{n}:=f_{n}(1+2\cos\theta)$.
By the equation \eqref{seq:rr-b1}, we have a recurrence relation
\begin{align}\label{seq:rr-b2}
x_{n+1}=2x_{n}\cos\theta-x_{n-1},
\end{align}
where $x_{1}=1+2\cos\theta$ and $x_{2}=1+2\cos\theta+2\cos2\theta$.
By solving the recurrence relation \eqref{seq:rr-b2}, we obtain the following equation.
\begin{align}
(e^{i\theta}-e^{-i\theta})x_{n}
&=2i(x_{2}\sin(n-1)\theta-x_{1}\sin(n-2)\theta)\notag\\
&=2i\left(2\sin\left(\frac{2(n+1)\theta}{2}\right)\cos\frac{\theta}{2}\right),\notag
\end{align}
where $i$ is the imaginary unit and $\displaystyle e^{x}:=\sum_{n=0}^{\infty}\frac{x^{n}}{n!}$. 
Since $e^{i\theta}-e^{-i\theta}=2i\sin\theta$ holds, we have 
\begin{align}
\sin\theta \cdot f_{n}(1+2\cos\theta)=2\sin\left(\frac{2(n+1)\theta}{2}\right)\cos\frac{\theta}{2}. \notag
\end{align}
Thus the all roots of the polynomial $f_{n}(x)$ are given by
\begin{align}
x=1+2\cos \left(\frac{2k\pi}{2n+1}\right)\quad(k=1,2,\ldots,n).\notag
\end{align}
This implies that $\rho(Q_{A})=1+2\cos(\frac{2\pi}{2n+1})$.
Similarly, we have the spectral radius for $\mathsf{X}_{n}=\mathsf{C}_{n}$.
The proof is complete.
\end{proof}

\subsection*{Acknowledgements}
The authors would like to thank Toshitaka Aoki for helpful comments on generalized preprojective algebras. 
The authors would also like to thank Sota Asai for sharing his result related to Proposition \ref{prop:brick-selfext}(1).

\end{document}